\documentclass[12pt,oneside,reqno]{amsart}

\usepackage{amsmath,amssymb,amsthm,textcomp}
\usepackage{amsfonts,graphicx}
\usepackage[mathscr]{eucal}
\usepackage{color,tikz}
\usepackage{csquotes}
\usepackage{mwe}
\usepackage{subcaption}
\usepackage[backend=bibtex,%
firstinits=true,%
doi=false,%
isbn=true,%
url=false,%
maxnames=99]{biblatex}%
\makeatletter
\@namedef{subjclassname@2020}{\textup{2020} Mathematics Subject Classification}
\makeatother
\AtEveryBibitem{\clearfield{issn}}
\AtEveryCitekey{\clearfield{issn}}
\numberwithin{equation}{section}
\DeclareNameAlias{sortname}{last-first}
\theoremstyle{definition}
\usepackage{mathtools}
\numberwithin{equation}{section}

\newcommand{\ncom}{\newcommand}

\ncom{\beq}{\begin{equation}}
\ncom{\eeq}{\end{equation}}
\ncom{\bea}{\begin{eqnarray*}}
\ncom{\eea}{\end{eqnarray*}}
\ncom{\beqa}{\begin{eqnarray}}
\ncom{\eeqa}{\end{eqnarray}}
\ncom{\nno}{\nonumber}
\ncom{\non}{\nonumber}
\ncom{\ds}{\displaystyle}
\ncom{\half}{\frac{1}{2}}
\ncom{\mbx}{\makebox{.25cm}}
\ncom{\hs}{\mbox{\hspace{.25cm}}}
\ncom{\rar}{\rightarrow}
\ncom{\Rar}{\Rightarrow}
\ncom{\noin}{\noindent}
\ncom{\bc}{\begin{center}}
\ncom{\ec}{\end{center}}
\ncom{\sz}{\scriptsize}
\ncom{\rf}{\ref}
\ncom{\s}{\sqrt{2}}
\ncom{\sgm}{\sigma}
\ncom{\Sgm}{\Sigma}
\ncom{\psgm}{\sigma^{\prime}}
\ncom{\dt}{\delta}
\ncom{\Dt}{\Delta}
\ncom{\lmd}{\lambda}
\ncom{\Lmd}{\Lambda}
\ncom{\Th}{\Theta}
\ncom{\e}{\eta}
\ncom{\eps}{\epsilon}
\ncom{\pcc}{\stackrel{P}{>}}
\ncom{\lp}{\stackrel{L_{p}}{>}}
\ncom{\dist}{{\rm\,dist}}
\ncom{\sspan}{{\rm\,span}}
\ncom{\re}{{\rm Re\,}}
\ncom{\im}{{\rm Im\,}}
\ncom{\sgn}{{\rm sgn\,}}
\ncom{\ba}{\begin{array}}
\ncom{\ea}{\end{array}}
\ncom{\hone}{\mbox{\hspace{1em}}}
\ncom{\htwo}{\mbox{\hspace{2em}}}
\ncom{\hthree}{\mbox{\hspace{3em}}}
\ncom{\hfour}{\mbox{\hspace{4em}}}
\ncom{\vone}{\vskip 2ex}
\ncom{\vtwo}{\vskip 4ex}
\ncom{\vonee}{\vskip 1.5ex}
\ncom{\vthree}{\vskip 6ex}
\ncom{\vfour}{\vspace*{8ex}}
\ncom{\norm}{\|\;\;\|}
\ncom{\integ}[4]{\int_{#1}^{#2}\,{#3}\,d{#4}}
\ncom{\vspan}[1]{{{\rm\,span}\{ #1 \}}}
\ncom{\dm}[1]{ {\displaystyle{#1} } }
\ncom{\ri}[1]{{#1} \index{#1}}

\newtheorem{theorem}{\bf Theorem}[section]
\newtheorem{remark}{\bf Remark}[section]
\newtheorem{proposition}{Proposition}[section]

\newtheoremstyle
    {remarkstyle}
    {}
    {11pt}
    {}
    {}
    {\bfseries}
    {:}
    {     }
    {\thmname{#1} \thmnumber{#2} }

\theoremstyle{remarkstyle}

\def\eps{\varepsilon}

\begin{document}
\title{Time-changed generalized fractional Skellam process}
\author[Mostafizar Khandakar]{Mostafizar Khandakar}
\address{Mostafizar Khandakar, Department of Science and Mathematics, Indian Institute of Information Technology Guwahati, Bongora, Assam 781015, India.}
\email{mostafizar@iiitg.ac.in}
\author[Bratati Pal]{Bratati pal}
\address{Bratati pal, Department of Science and Mathematics, Indian Institute of Information Technology Guwahati, Bongora, Assam 781015, India.}
\email{bratati.pal@iiitg.ac.in}
\author[Palaniappan Vellaisamy]{Palaniappan Vellaisamy}
\address{Palaniappan Vellaisamy, Department of Statistics and Applied Probability, University of California, Santa Barbara, CA 93106, USA}
\email{pvellais@ucsb.edu}
\subjclass[2020]{Primary: 60G22; Secondary: 60G55}
\keywords{Skellam process; time-changed processes; L\'evy subordinator; inverse subordinator; long-range dependence property}
\date{\today}
\begin{abstract}
In this paper, we introduce and study two time-changed variants of the
generalized fractional Skellam process. These are obtained by time-changing the generalized fractional Skellam process with
an independent L\'evy subordinator with finite moments of any order and its inverse, respectively. We call  the introduced processes the time-changed generalized fractional Skellam process-I (TCGFSP-I) and the time-changed generalized fractional Skellam process-II (TCGFSP-II), respectively. The probability
generating function, moment generating function, moments, factorial moments, variance, covariance, {\it etc.}, are derived for the TCGFSP-I. We obtain a variant of the law of the iterated logarithm
for it and establish its long-range dependence property. Several special cases of the TCGFSP-I are
considered, and the associated system of governing differential equations is obtained.
Later, some distributional properties and particular cases are discussed for the TCGFSP-II.
\end{abstract}

\maketitle
\section{Introduction}

The generalized counting process (GCP), introduced by Di Crescenzo {\it et al.} (2016) is a L\'evy process. We denote it by $\{M(t)\}_{t\ge0}$ and it generalizes the Poisson process. It is a non-renewal process and it performs $k$ kinds of jumps of amplitude $1,2,\dots,k$ with positive rates $\lambda_{1}, \lambda_{2},\dots,\lambda_{k}$, respectively. Its transition probabilities in an infinitesimal
interval of length $h$ are given by 
\begin{equation*}
	\mathrm{Pr}\{M(t+h)=n|M(t)=m\}=\begin{cases*}
		1-\Lambda h+o(h),\ \  n=m,\\
		\lambda_{j}h+o(h), \ \ n=m+j,\ \ j=1,2,\dots,k,\\
		o(h),\ \ n>m+k,
	\end{cases*}
\end{equation*}
where $\Lambda=\lambda_{1}+\lambda_{2}+\dots+\lambda_{k}$ for a fixed positive integer $k$ and $o(h)\to 0$ as $h\to 0$. For $k=1$, it reduces to the Poisson process. 

A time-changed variant of the GCP, namely, the generalized fractional counting process (GFCP) $\{M^{\alpha}(t)\}_{t\ge0}$, $0<\alpha\le 1$ is also introduced by Di Crescenzo {\it et al.} (2016). It is obtained by time-changing the GCP by an independent inverse stable subordinator $\{Y_{\alpha}(t)\}_{t\ge0}$. That is,	$M^{\alpha}(t)=M\big(Y_{\alpha}(t)\big)$
where the GCP is independent of the inverse stable subordinator.  

For $\alpha=1$, the GFCP reduces to the GCP. Moreover, for $k=1$, it reduces to the time-changed variant of the Poisson process, namely, the time fractional Poisson process (see Mainardi {\it et al.} (2004), Beghin and
Orsingher (2009), Meerschaert {\it et al.} (2011)). For other time-changed variants of the Poisson process such as the space-fractional Poisson process and the space-time fractional Poisson process, we refer the reader to Orsingher and Polito (2012).  Kataria and Khandakar (2022a) have shown that many known counting processes such as the Poisson process of order $k$, the P\'olya-Aeppli process of order $k$, the negative binomial process, the convoluted Poisson process (see Kataria and Khandakar (2021)), {\it etc.} and their fractional variants  are special cases of the GFCP.

Recently,  Kataria and Khandakar (2022b) introduced and studied a Skellam type variant of the GCP, namely, the generalized Skellam process (GSP) $\{\mathcal{S}(t)\}_{t\ge0}$. It is defined as  
\begin{equation*}
	\mathcal{S}(t)\coloneqq M_{1}(t)-M_{2}(t),
\end{equation*}
where  $\{M_{1}(t)\}_{t\ge0}$ and $\{M_{2}(t)\}_{t\ge0}$ are two generalized counting processes with positive rates $\lambda_{j}$'s and $\mu_{j}$'s, $j=1,2,\dots,k$, respectively and they are independent of each other. It is an integer valued L\'evy process. For $k=1$, it reduces to the Skellam process (see Barndorff-Nielsen {\it et al.} (2012)). For $\lambda_{j}=\lambda$ and $\mu_{j}=\mu$, $j=1,2,\dots,k$, it reduces to the Skellam process of order $k$ (see Gupta {\it et al.} (2020a)). The state probabilities $p(n,t)=\mathrm{Pr}\{\mathcal{S}(t)=n\}$ of GSP satisfy the following system of differential equations:
\begin{equation}\label{digsp}
	\frac{\mathrm{d}}{\mathrm{d}t}p(n,t)=\Lambda\big(p(n-1,t)-p(n,t)\big)-\bar{\Lambda}\big(p(n,t)-p(n+1,t)\big),\ \ n\in \mathbb{Z},
\end{equation}
with initial conditions $p(0,0)=1$ and $p(n,0)=0$, $n\neq 0$. Here, $\bar{\Lambda}=\mu_{1}+\mu_{2}+\dots+\mu_{k}$.

 Kataria and Khandakar (2022b) considered a time-changed variant of the GSP, namely, the generalized fractional Skellam process (GFSP). Here, we denote it by $\{\mathcal{S}^{\alpha}(t)\}_{t\ge0}$, $0<\alpha\leq 1$. It is defined as 
\begin{equation}\label{def}
	\mathcal{S}^{\alpha}(t)\coloneqq
	\begin{cases*}
		\mathcal{S}(Y_{\alpha}(t)),\ \ 0<\alpha<1,\\
		\mathcal{S}(t),\ \ \alpha=1,
	\end{cases*}
\end{equation}
where the GSP is independent of the inverse stable subordinator. It is important to note that the GFSP is a Skellam type varint of the GFCP. The GFSP exhibits the long-range dependence (LRD) property whereas its increment
process has the short-range dependence (SRD) property. For $k=1$, the GFSP reduces to the fractional Skellam process of type II (see Kerss {\it et al.} (2014)). For $\lambda_{j}=\lambda$ and $\mu_{j}=\mu$, $j=1,2,\dots,k$, the GFSP reduces to the fractional Skellam process of order $k$ (see Kataria and Khandakar (2024)). Recently, Tathe and Ghosh (2025a) introduced and studied the non-homogeneous version of the GSP and the GFSP. For more details and some other generalizations of the GSP and the GFSP, we refer the reader to Cinque and Orsingher (2025), Vishwakarma (2025), Tathe and Ghosh (2025b), {\it etc}.

Note that the recently studied time-changed processes are mainly
obtained by time-changing a point process with an independent subordinator and
its inverse.  A subordinator \( \{D_f(t)\}_{t\geq0} \) is a one-dimensional L\'evy process whose sample paths are non-decreasing
and $D_f(0)=0$ almost surely (a.s.). Its Laplace transform is given by (see Applebaum (2009), Section 1.3.2)
\begin{equation}\label{lexp}
	\mathbb{E}(e^{-sD_f(t)}) = e^{-tf(s)}, \ \  s > 0,
\end{equation}
where
\begin{equation*}
	f(s) = c_1 + c_2 s + \int_0^{\infty} (1 - e^{-sx}) \bar{\nu}(\mathrm{d}x),  \ \ c_1 \geq 0,  \ \ c_2 \geq 0,
\end{equation*}
is called the Bern\v{s}tein function. Here, \(  \bar{\nu}(\cdot) \) is a non-negative L\'evy measure on \((0,\infty)\) that satisfies $  \displaystyle\int_0^{\infty} (x \wedge 1) \bar{\nu}(\mathrm{d}x) < \infty$. The first passage time of a  subordinator, that is, 
\begin{equation*}
	H_f(t)\coloneqq \inf \{x \geq 0 : D_f(x) > t\}, \ \  t \geq 0,
\end{equation*}
is called the inverse subordinator.

Orsingher and Toaldo (2015) introduced and studied  a class of point processes which is obtained by time-changing the Poisson process with an independent
L\'evy subordinator. It includes many processes such as the space-fractional Poisson process,  negative binomial process, tempered fractional
Poisson process, {\it etc.}, as particular cases. Maheshwari and Vellaisamy (2019)
introduced  two time-changed variants of the time fractional Poisson process where the time-changes are done with an independent
L\'evy subordinator and its inverse. Recently, similar studies have been considered for the space-time fractional Poisson process (see Kataria and Khandakar (2022c)), the GFCP (see Kataria and Khandakar (2022b)) and the fractional Skellam process of order $k$ (see Kataria and Khandakar (2024)).

In this paper, our aim is to explore some time-changed variants of the GSP and GFSP. We give some preliminary results in Section \ref{section2} that will be used later. In Section \ref{3}, we introduce a time-changed variant of the GFSP, namely, the time-changed generalized fractional Skellam process-I (TCGFSP-I) by time-changing the GFSP with an independent L\'evy subordinator whose moments are finite. For $\alpha=1$, the TCGFSP-I reduces to  a time-changed version of the GSP, namely, the time-changed  generalized  Skellam process-I (TCGSP-I). We obtain the probability mass function (pmf) of TCGSP-I. The probability generating function (pgf), factorial moments,  variance, covariance, {\it etc.} of the TCGFSP-I are derived. Also, we obtain the moment generating function (mgf) and the moments of GFSP, using which the mgf and the moments of TCGFSP-I are derived. It is shown that the one-dimensional distributions of TCGFSP-I are not infinitely divisible and it has the overdispersion property. We establish a version of the law of iterated logarithm for it. We have shown that under certain conditions on the L\'evy subordinator, the TCGFSP-I exhibits the LRD property. It is important to note that the processes that exhibit the LRD property have applications in several areas such as finance, econometrics, hydrology, internet data traffic modeling, {\it etc.} Some particular cases of
TCGFSP-I are discussed by taking specific L\'evy subordinator such as the gamma subordinator,
tempered stable subordinator and inverse Gaussian subordinator. We derive
the system of differential equations that governs the state probabilities of these particular cases. Moreover, we have proved that the increment process of GFSP time-changed by a gamma subordinator has the SRD property.

In Section \ref{4}, another time-changed version of the GFSP, namely, the time-changed
generalized fractional Skellam process-II (TCGFSP-II) is considered. It
is obtained by time-changing the GFSP by an independent inverse subordinator. We obtain the mean, variance, covariance, pgf, {\it etc.} for the TCGFSP-II. For $\alpha=1$, the TCGFSP-II reduces to the  time-changed version of the GSP, namely, the 
generalized Skellam process-II (TCGSP-II). We derive the system of differential equation that governs the marginal distributions of TCGSP-II.
Some
particular cases of the TCGSP-II are also discussed.

\section{Preliminaries}\label{section2}
Here, we give some preliminary details about Mittag-Leffler function, L\'evy subordinator and its inverse, fractional derivatives, stable and inverse stable subordinator, the GSP and its time-changed variant, and the definition of LRD property. 

\subsection{Mittag-Leffler function}
The three-parameter Mittag-Leffler function is defined as (see Kilbas {\it et al.} (2006), Eq. (1.9.1))
\begin{equation}\label{mitag}
	E_{\alpha,\beta}^{\gamma}(x)\coloneqq\frac{1}{\Gamma(\gamma)}\sum_{j=0}^{\infty} \frac{\Gamma(j+\gamma)x^{j}}{j!\Gamma(j\alpha+\beta)},\ \ x\in\mathbb{R},
\end{equation}
where $\alpha>0$, $\beta>0$ and $\gamma>0$. It reduces to two-parameter Mittag-Leffler function for $\gamma=1$ and it further reduces to Mittag-Leffler function for $\gamma=\beta=1$.

\subsection{Subordinator and its inverse}
First, recall the following definition:

We say that a positive function $f(t)$ is asymptotically equal to a positive function $g(t)$, written as $f(t)\sim g(t)$, as $t\to \infty$, if $\lim\limits_{t\to \infty}f(t)/g(t) =1$.

Assume that $ 0 < \alpha < 1 $, $ 0 < s \le t $ and $ s $ be fixed. Then, the following limiting result hold as $ t \to \infty $ (see Maheshwari and Vellaisamy (2019), Theorem 3.3):
\begin{equation}\label{  Levyasymbeta}
	\alpha \mathbb{E}(D_f^{2\alpha}(t))B(\alpha, 1+\alpha; D_f(s)/D_f(t))\sim \mathbb{E}(D_f^\alpha(s)) \mathbb{E}(D_f^\alpha(t-s)),
\end{equation}
where $\mathbb{E}(D_f^\alpha(t)) <\infty $ for $ 0< t< \infty$ and  $\mathbb{E}(D_f^\alpha(t))  \to \infty $ as $ t \to \infty$. Here, $B(\alpha,\alpha+1;D_f(s)/D_f(t))$ is the incomplete beta function. 

Toaldo (2015) studied a convolution-type derivative  with respect to Bern\v{s}tein functions which generalize the Caputo
fractional derivative and Riemann-Liouville fractional derivative. The generalized Caputo derivative for an absolutely continuous function $w(\cdot)$ is defined as follows (see Toaldo (2015), Definition 2.4):
\begin{equation}\label{gcaputo}
	{}^{f}\mathscr{D}_{t} w(t) \coloneqq c_2 \frac{\mathrm{d}}{\mathrm{d}t} w(t) + \int_{0}^{t} \frac{\partial}{\partial t} w(t-s) \nu(s) \mathrm{d}s,
\end{equation}
where  $\nu(s)=c_1+\bar{\nu}(s, \infty)$ is the  tail of the L\'evy measure $\bar{\nu}(\cdot)$ associated with the Bern\v{s}tein function $f$. 

The generalized Caputo derivative is related with
the generalized Riemann-Liouville derivative $ {}^{f}\mathbb{D}_t $ as follows (see Toaldo (2015), Proposition 2.7)
\begin{equation}\label{2.27}
	{}^{f}\mathbb{D}_t w(t) = {}^{f}\mathscr{D}_t w(t) + \nu(t) w(0).
\end{equation}

For $\gamma\geq 0$, the Riemann-Liouville fractional derivative is defined as (see Kilbas {\it et al.} (2006))
\begin{equation}\label{RLd}
	\mathbb{D}_t^{\gamma}w(t):=\left\{
	\begin{array}{ll}
		\dfrac{1}{\Gamma{(m-\gamma)}}\displaystyle \frac{\mathrm{d}^{m}}{\mathrm{d}t^{m}}\int^t_{0} \frac{w(s)}{(t-s)^{\gamma+1-m}}\,\mathrm{d}s,\ \ m-1<\gamma<m,\\\\
		\displaystyle\frac{\mathrm{d}^{m}}{\mathrm{d}t^{m}}w(t),\ \ \gamma=m,
	\end{array}
	\right.
\end{equation}
where $m$ is a positive integer.

 The density function $l_f(x,t) $ of $\{	H_f(t)\}_{t\ge0}$ satisfies (see Toaldo (2015), Theorem 4.1)
\begin{equation}\label{2.28}
	{}^{f}\mathbb{D}_t l_f(x,t) = -\frac{\partial}{\partial x} l_f(x,t),
\end{equation}
with
\begin{equation*}
	l_f(x/c_2,t) = 0, \ \ l_f(0,t) = \nu(t), \ \ l_f(x,0) = \delta(x),
\end{equation*}
provided that the following condition holds:

Condition \textbf{I}. 
$ \bar{\nu}(0, \infty) = \infty$ and   $ \nu(s) = c_1 + \bar{\nu}(s, \infty)$ is absolutely continuous.

\subsection{Stable and inverse stable subordinator}
A stable subordinator $\{D_{\alpha}(t)\}_{t\ge0}$, $0<\alpha<1$, is a non-decreasing  L\'evy process with Bern\v{s}tein function $f(s)=s^\alpha$. Its first passage time is defined as 
\begin{equation*}
	Y_{\alpha}(t)\coloneqq \inf\{x\ge0:D_{\alpha}(x)>t\}.
\end{equation*}
The process $\{Y_{\alpha}(t)\}_{t\ge0}$ is called the inverse stable subordinator and its Laplace transform is given by
\begin{equation}\label{lpinv}
	\mathbb{E}\big(e^{-sY_{\alpha}(t)}\big)=E_{\alpha,1}({-st^{\alpha}}).
\end{equation}
It is a self-similar process (see Meerschaert and Scheffler (2004)), that is,
\begin{equation}\label{selfsimi}
	Y_{\alpha}(t)\overset{d}{=}t^{\alpha}Y_{\alpha}(1),
\end{equation} 
where $\overset{d}{=}$ denotes the equality in distribution. The mean and covariance of $\{Y_{\alpha}(t)\}_{t\ge0}$ are given by (see Leonenko {\it et al.} (2014))
\begin{align}\label{meani}
	\mathbb{E}\big(Y_{\alpha}(t)\big)&=\frac{t^{\alpha}}{\Gamma(\alpha+1)},\\
	\operatorname{Cov}\big(Y_{\alpha}(s),Y_{\alpha}(t)\big)&=\frac{1}{\Gamma^2(\alpha+1)}\big( \alpha s^{2\alpha}B(\alpha,\alpha+1)+F(\alpha;s,t)\big), \ \ 0<s\le t \label{covin},
\end{align}
where $F(\alpha;s,t)=\alpha t^{2\alpha}B(\alpha,\alpha+1;s/t)-(ts)^{\alpha}$. Here, $B(\alpha,\alpha+1)$  and $B(\alpha,\alpha+1;s/t)$ denote the beta function and the incomplete beta function, respectively. 
 
\subsection{The GSP and its time-changed variant}
Here, we give some known results for the GSP and its time-changed variant, namely, the GFSP (see Kataria and Khandakar (2022b)). 

The state probabilities of GSP are given by
\begin{equation}\label{pmfgsp}
	p(n,t)=e^{-(\Lambda+\bar{\Lambda})t}\big(\Lambda/\bar{\Lambda}\big)^{n/2}I_{|n|}\big(2t\sqrt{\Lambda \bar{\Lambda}}\big), \ \ n\in\mathbb{Z}, 
\end{equation}
where $I_{n}(\cdot)$ is the modified Bessel function of first kind and it is defined as (see Sneddon (1956), p. 114)
\begin{equation*}
	I_{n}(x)=\sum_{j=0}^{\infty}\frac{(x/2)^{2j+n}}{j!(j+n)!},\  \ x\in\mathbb{R}.
\end{equation*} 

The mgf  $\mathcal{M}_{\mathcal{S}}(u,t)=\mathbb{E}(e^{u\mathcal{\mathcal{S}}(t)})$ of GSP is given by (see Tathe and Ghosh (2025a), Eq. (3.14))
\begin{equation}\label{mgfgsp}
	\mathcal{M}_{\mathcal{S}}(u,t)=\exp{\Big(\sum_{j=1}^{k}\big(\lambda_{j}(e^{uj}-1)+\mu_{j}(e^{-uj}-1)\big)t \Big)},\ \ u \in \mathbb{R}.
\end{equation}
The following limiting result holds for it:
\begin{equation}\label{limitS}
	\lim\limits_{t\to\infty}\frac{\mathcal{S}(t)}{t}=\sum_{j=1}^{k}j(\lambda_{j}-\mu_{j}),\ \ \text{in probability}.
\end{equation}
The pgf $G_{\mathcal{S}^\alpha}(u,t)=\mathbb{E}(u^{\mathcal{S}^\alpha(t)})$ of GFSP  is given by
\begin{equation}\label{pgfgfsp}
	G_{\mathcal{S}^\alpha}(u,t)=	E_{\alpha,1}\Big(\sum_{j=1}^{k}\Big(\lambda_{j}(u^{j}-1)+\mu_{j}(u^{-j}-1)\Big)t^{\alpha}\Big), \ \ |u| \le 1.
\end{equation}
The mean and  covariance of GFSP are given by
\begin{align}
	\mathbb{E}\big(\mathcal{S}^{\alpha}(t)\big)&=m_{1}\mathbb{E}\big(Y_{\alpha}(t)\big),\label{mgfsp}\\
	\operatorname{Cov}\big(\mathcal{S}^{\alpha}(s),\mathcal{S}^{\alpha}(t)\big)&=m_{2}\mathbb{E}\big(Y_{\alpha}(s)\big)+m_{1}^{2}\operatorname{Cov}\big(Y_{\alpha}(s),Y_{\alpha}(t)\big), \ \  0<s\leq t,\label{cgfsp}
\end{align}
where
\begin{equation}\label{r1r2}
	m_{1}=\sum_{j=1}^{k}j(\lambda_{j}-\mu_{j}) \ \  \text{and}\ \  m_{2}=\sum_{j=1}^{k}j^{2}(\lambda_{j}+\mu_{j}).
\end{equation}
\subsection{The LRD and SRD properties} 

The following definition of LRD and SRD property will be used (see  Maheshwari and Vellaisamy (2016)):

Let $s>0$ be fixed and $\{X(t)\}_{t\ge0}$ be a non-stationary stochastic process whose correlation function has the following asymptotic behaviour: 
\begin{equation*}
	\operatorname{Corr}(X(s),X(t))\sim c(s)t^{-\gamma},\ \ \text{as}\ t\to\infty,
\end{equation*}
for some $c(s)>0$. Then, the process $\{X(t)\}_{t\ge0}$ is said to exhibit the LRD property if $\gamma\in(0,1)$ and the SRD property if $\gamma\in (1,2)$.

\section{Time-changed generalized fractional Skellam process}\label{3}
In this section, we introduce a time-changed variant of the GFSP. We call it the time-changed generalized fractional Skellam process-I (TCGFSP-I) and denote it by $\{\mathcal{Z}_{f}^{\alpha}(t)\}_{t\ge0}$, $0<\alpha\le1$. It is obtained by time-changing the GFSP with an independent L\'evy subordinator $\{D_f (t)\}_{t\ge0}$  whose moments are finite, that is, $\mathbb{E}\big(D_f^r(t)\big) < \infty \text{ for all } r > 0 $. That is,  
\begin{equation}\label{qws11ww1}
	\mathcal{Z}_{f}^{\alpha}(t)\coloneqq \mathcal{S}^{\alpha}(D_f (t))=\mathcal{S}(Y_{\alpha}(D_f (t))),
\end{equation}
where the GFSP $\{\mathcal{S}^{\alpha}(t)\}_{t\ge0}$ is independent of $\{D_f (t)\}_{t\ge0}$.

For $\alpha=1$, the TCGFSP-I reduces to  a time-changed variant of the GSP, namely, the time-changed  generalized  Skellam process-I (TCGSP-I) $\{\mathcal{Z}_{f}(t)\}_{t\ge0}$, that is, 
\begin{equation}\label{zft}
	\mathcal{Z}_{f}(t)\coloneqq \mathcal{S}(D_f (t)).
\end{equation}

Note that the TCGFSP-I is non-Markovian and is also not a L\'evy process. Since it has the inverse stable subordinator $\{Y_\alpha(t)\}_{t\ge 0}$ as a time-changed component, which in general, is non-Markovian with non-stationary and non-independent increments (see Veillette and Taqqu (2010)). However, the TCGSP-I is a  L\'evy process since it is a composition of two independent L\'evy process.

For $\lambda_{j}=\lambda$ and $\mu_{j}=\mu$, $j=1,2,\dots,k$, the TCGFSP-I reduces to the time-changed fractional Skellam process of order $k$ (see Kataria and Khandakar (2024)), and the TCGSP-I reduces to the time-changed Skellam process of order $k$ (see Gupta {\it et al.} (2020a)). It is known that the convoluted fractional Poisson process is a limiting process of GFCP (see Kataria and Khandakar (2021)). Sengar and Upadhye (2022) introduced and studied the Skellam version of the convoluted fractional Poisson process, namely, the fractional convoluted Skellam process by taking the difference of two independent convoluted fractional Poisson processes with different intensities. Also, they consider the time-changed variant of the fractional convoluted Skellam process in which the time-change is done with an independent L\'evy subordinator. Thus, the time-changed fractional convoluted Skellam process can be obtained as a limiting process of the TCGFSP-I. Furthermore, for $k=1$, the TCGFSP-I and the TCGSP-I reduces to a time-changed variant of the fractional Skellam process of type II (see Kerss {\it et al.} (2014)) and the Skellam process, respectively. 
  
\begin{remark} 
	Lee and Macci (2023) showed that the Skellam process can be seen as a compound Poisson process. Similarly, it can be seen that the GSP is equal in distribution to the following compound Poisson process:
	\begin{equation*}
		\mathcal{S}(t)\overset{d}{=}\sum_{i=1}^{N(t)}X_{i}, \ \ t\ge0,
	\end{equation*}
	where $\{N(t)\}_{t\ge0}$ is a Poisson process with intensity $\Lambda+\bar{\Lambda}$ and it is independent of the sequence of independent and identically distributed random variables $\{X_{i}\}_{i\ge1}$ such that 
	\begin{equation*}
		\mathrm{Pr}\{X_{1}=j\}=\frac{\lambda_{j}}{\Lambda+\bar{\Lambda}} \ \text{and} \ \mathrm{Pr}\{X_{1}=-j\}=\frac{\mu_{j}}{\Lambda+\bar{\Lambda}},\ j=1,2,\dots,k.
	\end{equation*}
	Thus, it follows that
	\begin{equation*}
		\mathcal{S}^{\alpha}(t)\overset{d}{=}\sum_{i=1}^{N^{\alpha}(t)}X_{i},\ \ t\ge0,
	\end{equation*}
	where $\{N^{\alpha}(t)\}_{t\ge0}$ is a time fractional Poisson process with intensity $\Lambda+\bar{\Lambda}$.
	Thus, \eqref{zft} can be written as
	\begin{equation*}
		\mathcal{Z}_f(t)\overset{d}{=}\sum_{i=1}^{N(D_f(t))}X_{i},\ \ t\ge0,  
	\end{equation*}
	where $\{N(D_f(t))\}_{t\ge 0}$ is a Poisson process subordinated by an independent L\'evy subordinator which is introduced and studied by Orsingher and Toaldo (2015).
	Similarly, from \eqref{qws11ww1} we get
	\begin{equation*}
		\mathcal{Z}_f^\alpha(t)\overset{d}{=}\sum_{i=1}^{N^\alpha(D_f(t))}X_{i},\ \ t\ge0,
	\end{equation*}
	where $\{N^\alpha (D_f(t))\}_{t\ge 0}$ is a time-changed time fractional Poisson process (see Maheshwari and Vellaisamy (2019)).
\end{remark}
In the following result, we obtain the pmf of TCGSP-I.
\begin{theorem}\label{thpmfI}The state probabilities 
	$p_{f}(n,t)=\mathrm{Pr}\{\mathcal{Z}_{f}(t)=n\}$ of TCGSP-I are given by
	\begin{equation}\label{pmf_tcspok}
		p_{f}(n,t)=\sum_{m=\max(0,-n)}^{\infty} \frac{\Lambda^{m+n} \bar{\Lambda}^{m} }{(m+n)! m!} \mathbb{E}\big(e^{-(\Lambda+\bar{\Lambda})D_{f}(t)} D_{f}^{2m+n}(t)\big), \ \ n \in \mathbb{Z}.
	\end{equation}
\end{theorem}
\begin{proof} 
	Let $h_{f}(x,t)$ be the probability density function (pdf) of $\{D_f (t)\}_{t\ge 0}$ and recall that $p(n,t)$ is the pmf of GSP  given in \eqref{pmfgsp}. From (\ref{zft}), we can write 
	\begin{align*}
		p_{f}(n,t)&=\int_{0}^{\infty}p(n,x)h_{f}(x,t) \mathrm{d}x\\
		&=\int_{0}^{\infty}e^{-(\Lambda+\bar{\Lambda})x}{\big(\Lambda/\bar{\Lambda}\big)}^{n/2}I_{|n|}\big(2x\sqrt{\Lambda \bar{\Lambda}}\big)h_{f}(x,t)\mathrm{d}x\\
		&=\sum_{m=\max(0,-n)}^{\infty} \frac{\Lambda^{m+n} \bar{\Lambda}^{m} }{(m+n)! m!} \int_{0}^{\infty} e^{-(\Lambda+\bar{\Lambda})x}x^{2m+n} h_{f}(x,t) \mathrm{d}x\\
		&=\sum_{m=\max(0,-n)}^{\infty} \frac{\Lambda^{m+n} \bar{\Lambda}^{m} }{(m+n)! m!} \mathbb{E}\big(e^{-(\Lambda+\bar{\Lambda})D_{f}(t)} D_{f}^{2m+n}(t)\big).
	\end{align*}
	This completes the proof.
\end{proof}
\begin{remark}
	On substituting $\lambda_{j}=\lambda$, $\mu_{j}=\mu$, $j=1,2,\dots,k$ in (\ref{pmf_tcspok}), we get
	\begin{equation*}
		p_{f}(n,t)\big|_{\lambda_{j}=\lambda, \mu_{j}=\mu}=\sum_{m=\max(0,-n)}^{\infty} \frac{(k\lambda)^{m+n} (k\mu)^{m} }{(m+n)! m!} \mathbb{E}\big(e^{-k(\lambda+\mu)D_{f}(t)} D_{f}^{2m+n}(t)\big),\ \ n\in \mathbb{Z},
	\end{equation*}
	which agrees  with the pmf of time-changed Skellam process of order $k$ (see Gupta {\it et al.} (2020a), Eq. (47)).
\end{remark}
Now, we obtain the pgf $G^\alpha_{f}(u,t)=\mathbb{E}(u^{\mathcal{Z}_f^\alpha(t)})$ of TCGFSP-I.
\begin{theorem}\label{thpgfI}
	The pgf of TCGFSP-I is given by \begin{equation}\label{pgffsp}
		G^\alpha_{f}(u,t) =\sum_{n=0}^{\infty}\frac{\big(\sum_{j=1}^{k}\left(\lambda_{j}(u^{j}-1)+\mu_{j}(u^{-j}-1)\right)\big)^n}{\Gamma(n\alpha +1 )}\mathbb{E}\big((D_f(t))^{n\alpha }\big),\ \ |u|\le1.
	\end{equation}
\end{theorem}
\begin{proof}
	From (\ref{qws11ww1}), we have
	\begin{align*}
		G^\alpha_{f}(u,t)&=\int_{0}^{\infty}G_{\mathcal{S}^\alpha}(u,x)h_{f}(x,t)\mathrm{d}x\\
		&=\int_{0}^{\infty}E_{\alpha,1}\Big(\sum_{j=1}^{k}\left(\lambda_{j}(u^{j}-1)+\mu_{j}(u^{-j}-1)\right)x^\alpha\Big)h_{f}(x,t)\mathrm{d}x, \ \ \text{(using \eqref{pgfgfsp})}\\
		&=\sum_{n=0}^{\infty}\frac{\big(\sum_{j=1}^{k}\left(\lambda_{j}(u^{j}-1)+\mu_{j}(u^{-j}-1)\right)\big)^n}{\Gamma(n\alpha +1 )}\int_{0}^{\infty}x^{n\alpha }h_{f}(x,t)\mathrm{d}x,
	\end{align*}
	which gives the required result.
\end{proof}
\begin{remark}\label{pgftc}
On substituting $\alpha=1$ in \eqref{pgffsp}, we get the pgf $G_{f}(u,t)=\mathbb{E}(u^{\mathcal{Z}_f(t)})$ of TCGSP-I in the following form:\begin{align*}
		G_{f}(u,t) &=\sum_{n=0}^{\infty}\Big(-\sum_{j=1}^{k}\left(\lambda_{j}(1-u^{j})+\mu_{j}(1-u^{-j})\right)\Big)^n \mathbb{E}\big((D_f(t))^{ n}\big)/n!\\
		&=\mathbb{E}\Big(e^{-\sum_{j=1}^{k}\big(\lambda_{j}(1-u^{j})+\mu_{j}(1-u^{-j})\big)D_f(t)}\Big) \\
		&=\exp\bigg(-tf\Big(\sum_{j=1}^{k}(\lambda_{j}(1-u^{j})+\mu_{j}(1-u^{-j}))\Big)\bigg),
	\end{align*}
where in the last step we have used \eqref{lexp}.

 Thus, the pgf of TCGSP-I satisfies the following differential equation:
\begin{equation*}
	\frac{\mathrm{d}}{\mathrm{d}t}G_{f}(u,t)=-f\Big(\sum_{j=1}^{k}(\lambda_{j}(1-u^{j})+\mu_{j}(1-u^{-j}))\Big)G_{f}(u,t), \ \ G_{f}(u,0)=1.
\end{equation*}

\end{remark}
Next, we derive the factorial moments of TCGFSP-I by using its pgf. 
\begin{proposition}\label{pp3.1}
	The $r$th factorial moment $\Psi_f^\alpha(r, t) = \mathbb{E}\big(\mathcal{Z}^\alpha_f(t) (\mathcal{Z}^\alpha_f(t) - 1) \cdots (\mathcal{Z}^\alpha_f(t) - r + 1)\big)$, $r \geq 1$ of TCGFSP-I is given by
	\begin{equation*}
		\Psi_f^\alpha(r, t) =r! \sum_{n=1}^{r}\frac{\mathbb{E}\big((D_f(t))^{n\alpha}\big)}{\Gamma( n \alpha +1 )}\underset{m_i\in\mathbb{N}}{\underset{\sum _{i=1}^n m_{i}=r}{\sum}}\prod_{l=1}^n\Big(\frac{1}{m_l!}\Big(\sum_{j=1}^k \big((j)_{m_l} \lambda_j + (-1)^{m_l} j^{(m_l)} \mu_j\big)\Big)\Big),
	\end{equation*}
	where $j^{(m_l)} = j(j+1) \cdots (j+m_l-1)$ denotes the rising factorial and $(J)_{m_l} = j(j-1) \cdots (j-m_l+1)$ denotes the falling factorial.
\end{proposition}

\begin{proof}
	Let $\mathcal{G}(u) =\sum_{j=1}^k \big(\lambda_{j}(u^{j}-1)+\mu_{j}(u^{-j}-1)\big)$. From \eqref{pgffsp}, we have
	
	\begin{align}\label{eq0}
		\Psi_f^\alpha(r,t) &=\frac{\partial^r }{\partial u^r}{G}_f^\alpha (u,t)\Big|_{u=1}\nonumber\\ &=\sum_{n=0}^{\infty}\frac{\mathbb{E}\big((D_f(t))^{n\alpha }\big)}{\Gamma(n \alpha +1 )}\frac{\mathrm{d}^r}{\mathrm{d} u^r} (\mathcal{G}(u))^n \Big|_{u=1}.
	\end{align}
	Now, by using the following result (see Johnson (2002), Eq. (3.6)):
	\begin{equation*}
		\frac{\mathrm{d}^{r}}{\mathrm{d}u^{^{r}}}(f(u))^{n}=\underset{m_j\in\mathbb{N}_0}{\underset{m_{1}+m_{2}+\dots+m_{n}=r}{\sum}}\frac{r!}{m_1!m_2!\ldots m_n!}f^{(m_{1})}(u)f^{(m_{2})}(u)\dots f^{(m_{n})}(u),
	\end{equation*}
we can write
	\begin{align}\label{eq1}
		\frac{\mathrm{d}^r}{\mathrm{d} u^r} (\mathcal{G}(u))^n \Big|_{u=1} &= r!\underset{m_i\in\mathbb{N}_0}{\underset{\sum _{i=1}^n m_{i}=r}{\sum}} \prod_{l=1}^n \frac{1}{m_l!}  \frac{\mathrm{d}^{m_l}}{\mathrm{d}u^{m_l}} \mathcal{G}(u) \Big|_{u=1} \nonumber\\ &= r!\underset{m_i\in\mathbb{N}}{\underset{\sum _{i=1}^n m_{i}=r}{\sum}} \prod_{l=1}^n \frac{1}{m_l!} \Big(\sum_{j=1}^k \big((j)_{m_l} \lambda_j + (-1)^{m_l} j^{(m_l)} \mu_j\big)\Big).
	\end{align}
Here, $\mathbb{N}_0$ is the set of non-negative integers.  From \eqref{eq0} and \eqref{eq1}, we get
	\begin{align}\label{eq2}
		\Psi^\alpha_f(r,t)&=r! \sum_{n=0}^{\infty}\frac{\mathbb{E}\big((D_f(t))^{n \alpha }\big)}{\Gamma(n \alpha +1 )}\underset{m_i\in\mathbb{N}}{\underset{\sum _{i=1}^n m_{i}=r}{\sum}} \prod_{l=1}^n \frac{1}{m_l!}\Big(\sum_{j=1}^k \big((j)_{m_l} \lambda_j + (-1)^{m_l} j^{(m_l)} \mu_j\big)\Big).
	\end{align}
	Note that the right hand side of \eqref{eq2} vanishes for $ n=0$ and $n \ge r+1$. This completes the proof.
	
\end{proof}

\begin{remark}  The mgf  $ \mathcal{M}_{ \mathcal{S}^{\alpha}}(u,t)=\mathbb{E}(e^{u \mathcal{S}^{\alpha}(t)})$ of GFSP can be evaluated as follows:
	
	Let $h_{\alpha}(x,t)$ be the pdf of $\{Y_\alpha(t)\}_{t\ge 0}$. From \eqref{def}, we have
	\begin{align}\label{mgfgfsp}
		\mathcal{M}_{ \mathcal{S}^{\alpha}}(u,t)&=\int_{0}^{\infty}\mathbb{E}(e^{u \mathcal{S}(x)})h_{\alpha}(x,t)\mathrm{d}x \nonumber\\
		&=\int_{0}^{\infty}\exp{\Big(\sum_{j=1}^{k}\big(\lambda_{j}(e^{uj}-1)+\mu_{j}(e^{-uj}-1)\big)x \Big)}h_{\alpha}(x,t)\mathrm{d}x, \ \ (\text{using \eqref{mgfgsp}})\nonumber\\&=E_{\alpha,1}\Big(\sum_{j=1}^{k}\big(\lambda_{j}(e^{uj}-1)+\mu_{j}(e^{-uj}-1)\big)t^\alpha \Big),\ \ u \in \mathbb{R},
	\end{align} 
	where in the last step we have used \eqref{lpinv}.
	
Since the Mittag-Leffler function is an eigenfunction of the Caputo fractional derivative, it follows that
	\begin{equation*}
		\frac{\mathrm{d}^\alpha}{\mathrm{d} t^\alpha} \mathcal{M}_{ \mathcal{S}^{\alpha}}(u,t) =\Big(\sum_{j=1}^{k}\big(\lambda_{j}(e^{uj}-1)+\mu_{j}(e^{-uj}-1)\big)\Big) \mathcal{M}_{ \mathcal{S}^{\alpha}}(u,t), \ \ \mathcal{M}_{ \mathcal{S}^{\alpha}}(u,0)=1,
	\end{equation*}
where	$\dfrac{\mathrm{d}^\alpha}{\mathrm{d} t^\alpha}$ is the Caputo fractional derivative (see Kilbas {\it et al.} (2006)).

On substituting $k=1$ in \eqref{mgfgfsp}, we get the mgf of fractional Skellam process of type II (see Kerss {\it et al.} (2014), Eq. (3.6)).
	
By using the mgf of GFSP, we can obtain its $r$th order moment and it is given by
	\begin{equation*}
		\mathbb{E}\big((\mathcal{S}^\alpha(t))^r \big) =r! \sum_{n=1}^{r}\frac{t^{n \alpha }}{\Gamma(n \alpha +1)}\underset{m_i\in\mathbb{N}}{\underset{\sum _{i=1}^n m_{i}=r}{\sum}} \prod_{l=1}^n\Big( \frac{1}{m_l!} \Big(\sum_{j=1}^k\big(j^{m_l} \lambda_j + (-1)^{m_l} j^{m_l} \mu_j \big)\Big)\Big),\ \ r \geq 1.
	\end{equation*}
	Its proof follows along the similar lines to that of Proposition 3 of Kataria and Khandakar (2022b).
\end{remark}   
Next, we obtain the mgf $\mathcal{M}^\alpha_{f}(u,t)=\mathbb{E}(e^{u\mathcal{Z}_f^\alpha(t)})$ of TCGFSP-I using \eqref{mgfgfsp} as follows:

 From \eqref{qws11ww1}, we have
\begin{align}\label{mgffsp}
	\mathcal{M}^\alpha_{f}(u,t) &=\int_{0}^{\infty}\mathcal{M}_{ \mathcal{S}^{\alpha}}(u,x)h_{f}(x,t)\mathrm{d}x \nonumber\\ &=\int_{0}^{\infty}E_{\alpha,1}\Big(\sum_{j=1}^{k}\big(\lambda_{j}(e^{uj}-1)+\mu_{j}(e^{-uj}-1)\big)x^\alpha \Big)h_{f}(x,t)\mathrm{d}x\nonumber\\
	&=\sum_{n=0}^{\infty}\frac{\big(\sum_{j=1}^{k}(\lambda_{j}(e^{uj}-1)+\mu_{j}(e^{-uj}-1))\big)^n}{\Gamma(n \alpha +1 )}\mathbb{E}\big((D_f(t))^{n \alpha }\big),\ \ u\in \mathbb{R}.
\end{align} 
Thus, by using \eqref{mgffsp}, we can obtain the $r$th order moment of TCGFSP-I and it is given by
\begin{equation*}
	\mathbb{E}\big((\mathcal{Z}^\alpha_f(t))^r \big) =r! \sum_{n=1}^{r}\frac{\mathbb{E}\big((D_f(t))^{n \alpha }\big)}{\Gamma(n \alpha +1)}\underset{m_i\in\mathbb{N}}{\underset{\sum _{i=1}^n m_{i}=r}{\sum}} \prod_{l=1}^n \Big(\frac{1}{m_l!}\Big(\sum_{j=1}^k \big(j^{m_l} \lambda_j + (-1)^{m_l} j^{m_l} \mu_j \big)\Big)\Big),\ \ \text{$r \geq 1$}.
\end{equation*}
Its proof follows along the similar lines to that of Proposition \ref{pp3.1}.

On substituting $\alpha=1$ in \eqref{mgffsp}, we get the mgf $\mathcal{M}_{f}(u,t)=\mathbb{E}(e^{u\mathcal{Z}_f(t)})$ of TCGSP-I in the following form: 
\begin{equation*}
	\mathcal{M}_{f}(u,t)=\exp\Big(-tf\Big(\sum_{j=1}^{k}\big(\lambda_{j}(1-e^{uj})+\mu_{j}(1-e^{-uj})\big)\Big)\Big),\ \ u \in \mathbb{R}.
\end{equation*}  
Its proof follows from a calculation analogous to that in Remark \ref{pgftc}. Thus, the mgf of TCGSP-I satisfies
\begin{equation*}
	\frac{\mathrm{d}}{\mathrm{d}t}\mathcal{M}_{f}(u,t)=-f\Big(\sum_{j=1}^{k}\big(\lambda_{j}(1-e^{uj})+\mu_{j}(1-e^{-uj})\big)\Big)\mathcal{M}_{f}(u,t),\ \ \mathcal{M}_{f}(u,0)=1.
\end{equation*}
\begin{remark}
On substituting $\lambda_j=\lambda$ and $\mu_j=\mu$, $j=1,2\dots, k$ in \eqref{mgffsp}, we get the mgf of a time-changed Skellam process of order $k$ (see Gupta {\it et al.} (2020a),  Section (6)).
\end{remark}

Next, we discuss the mean, variance and covariance of TCGFSP-I.

Let $0 < s \leq t < \infty$, and assume that 
$$l_1 =m_1/\Gamma(1 + \alpha),\  l_2 =m_2/\Gamma(1 + \alpha),\  d=\alpha l_1^2 B(\alpha, 1 + \alpha),$$ where $m_1$ and $m_2 $ are given in \eqref{r1r2}.

\noindent Using \eqref{meani} and \eqref{mgfsp}, the mean of TCGFSP-I is obtained as follows:
\begin{equation}\label{meantcgfsp}
	\mathbb{E}\big( \mathcal{Z}_f^{\alpha}(t) \big) = \mathbb{E}\big( \mathbb{E}( \mathcal{S}^{\alpha}(D_f(t)) | D_f(t)) \big) = l_1 \mathbb{E}\big( D_f^{\alpha}(t)\big).
\end{equation}
From \eqref{meani}, \eqref{covin} and \eqref{cgfsp}, we get
\begin{align*}
	\mathbb{E}\big( \mathcal{S}^{\alpha}(s) \mathcal{S}^{\alpha}(t) \big) = l_2 s^{\alpha} + ds^{2\alpha} + l_1^2 \alpha t^{2\alpha} B\big(\alpha,  \alpha +1; s/t \big).
\end{align*}
Thus,
\begin{align*}
	\mathbb{E}\big(\mathcal{Z}_f^{\alpha}(s) \mathcal{Z}_f^{\alpha}(t) \big) =& \mathbb{E}\big( \mathbb{E}\big( \mathcal{S}^{\alpha}(D_f(s))  \mathcal{S}^{\alpha}(D_f(t)) \big| D_f(s), D_f(t) \big) \big)\\
	=& l_2 \mathbb{E}\big( D_f^{\alpha}(s) \big) + d \mathbb{E}\big( D_f^{2\alpha}(s) \big) + l_1^2 \alpha\mathbb{E}\big( D_f^{2\alpha}(t) B\big(\alpha, \alpha +1 ; {D_f(s)}/{D_f(t)} \big) \big).
\end{align*}
Therefore, the covariance of TCGFSP-I is obtained in the following form:
\begin{align}\label{covtcgfsp}
	\text{Cov}\big(\mathcal{Z}_f^{\alpha}(s), \mathcal{Z}_f^{\alpha}(t) \big)& = l_2\mathbb{E}\big( D_f^{\alpha}(s) \big) + d \mathbb{E}\big( D_f^{2\alpha}(s) \big) - l_1^2 \mathbb{E}\big( D_f^{\alpha}(s) \big) \mathbb{E}\big( D_f^{\alpha}(t) \big) \nonumber \\& \hspace{2.0cm} +  l_1^2 \alpha \mathbb{E}\big( D_f^{2\alpha}(t) B\big(\alpha, \alpha +1; {D_f(s)}/{D_f(t)} \big) \big).
\end{align}
On substituting $s = t$ in \eqref{covtcgfsp}, we get the variance of TCGFSP-I and it is given by \begin{equation}\label{vartcgfsp}
	\text{Var}\big( \mathcal{Z}_f^{\alpha}(t) \big) =  \mathbb{E}\big( D_f^{\alpha}(t) \big) \big( l_2 - l_1^2\mathbb{E}\big( D_f^{\alpha}(t) \big) \big) + 2d \mathbb{E}\big( D_f^{2\alpha}(t) \big).
\end{equation}

\begin{remark} It is known that a stochastic process $\{X(t)\}_{t > 0}$  exhibits overdispersion property if $\text{Var}(X(t))-\mathbb{E}(X(t)) > 0$ for all $t>0$. Kataria and Khandakar (2022b) showed that the GSPP and GFSP are overdispersed processes. From \eqref{meantcgfsp} and \eqref{vartcgfsp}, we can write
	\begin{align*}
		\mathrm{Var}\big(\mathcal{Z}^{\alpha}_{f}(t)\big) - \mathbb{E}\big(\mathcal{Z}^{\alpha}_{f}(t)\big)&=\frac{m_1^2}{\alpha}\Big(\frac{\mathbb{E}\big(D^{2\alpha}_{f}(t)\big)}{\Gamma(2\alpha)}-\frac{\big( \mathbb{E}\big(D^{\alpha}_{f}(t)\big)\big)^2}{\alpha \Gamma ^2(\alpha)}\Big)+\mathbb{E}\big(D^{\alpha}_{f}(t)\big)(l_2-l_1)\\
		& \ge {m_1^2 \big( \mathbb{E}\big(D^{\alpha}_{f}(t)\big)\big)^2}\frac{1}{\alpha}\Big(\frac{1}{\Gamma(2\alpha)}-\frac{1}{\alpha \Gamma ^2(\alpha)}\Big)+\mathbb{E}\big(D^{\alpha}_{f}(t)\big)(l_2-l_1),
	\end{align*}
	where in the last step we have used $\mathbb{E}\big(D^{2\alpha}_{f}(t)\big) \ge \big( \mathbb{E}\big(D^{\alpha}_{f}(t)\big)\big)^2$. As $\frac{1}{\alpha}\big(\frac{1}{\Gamma(2\alpha)}-\frac{1}{\alpha \Gamma ^2(\alpha)}\big) > 0$ for all $\alpha \in (0,1)$  (see Beghin and Macci (2014), Section 3.1), the TCGFSP-I exhibits overdispersion.
\end{remark}

Next, we show that under certain asymptotic conditions on the L\'evy subordinator, the TCGFSP-I exhibits the LRD property.
\begin{theorem}\label{lrd}
Let $\{D_f(t)\}_{t\ge0}$ be a L\'evy subordinator such that 
	\begin{equation}\label{exp  Levyasym}
		\mathbb{E}(D_f^{i\alpha}(t)) \sim k_i t^{i\rho}, \ \ \text{for}\  i=1, 2, 
	\end{equation}
	for some $0 < \rho < 1$, $k_1>0$ and $k_2 \geq k_1^2$. Then, the TCGFSP-I exhibits the LRD property.
\end{theorem}
\begin{proof}
	 Let $0 < s < t$. By using \eqref{  Levyasymbeta} in \eqref{covtcgfsp} for fixed $s$ and large $t$, we get
\begin{align}\label{cov prs asym}
	\text{Cov}\big(\mathcal{Z}^{\alpha}_f(s), \mathcal{Z}^{\alpha}_f(t)\big)&\sim l_2 \mathbb{E}(D_f^\alpha(s)) + d\mathbb{E}(D_f^{2\alpha}(s)) - l_1^2\mathbb{E}(D_f^\alpha(s)) \mathbb{E}(D_f^\alpha(t)) \nonumber\\
	 & \ \ +l_1^2\mathbb{E}(D_f^\alpha(s)) \mathbb{E}(D_f^\alpha(t-s)) \nonumber \\& \sim l_2 \mathbb{E}(D_f^\alpha(s)) + d\mathbb{E}(D_f^{2\alpha}(s)) - l_1^2k_1\mathbb{E}(D_f^\alpha(s))(t^\rho -(t-s)^\rho)\nonumber\\
	& \sim l_2 \mathbb{E}(D_f^\alpha(s)) + d\mathbb{E}(D_f^{2\alpha}(s)) - l_1^2k_1\mathbb{E}(D_f^\alpha(s))\rho s t^{\rho-1},
\end{align}
where we have used $\eqref{exp  Levyasym}$ in the penultimate step. Now, we use \eqref{exp  Levyasym} in \eqref{vartcgfsp} to obtain
\begin{align*}
	\text{Var}\big(\mathcal{Z}^{\alpha}_f(t)\big)&\sim l_2k_1 t^\rho + (2dk_2 - l_1^2k_1^2)t^{2\rho}\\ &\sim (2dk_2 - l_1^2k_1^2)t^{2\rho}\\&=\frac{m_1^2}{\alpha}\Big(\frac{k_2}{\Gamma(2\alpha)}-\frac{k_1^2}{\alpha \Gamma ^2(\alpha)}\Big)t^{2\rho}= d_1t^{2\rho},
\end{align*}
where $d_1=\dfrac{m_1^2}{\alpha}\Big(\dfrac{k_2}{\Gamma(2\alpha)}-\dfrac{k_1^2}{\alpha \Gamma ^2(\alpha)}\Big) >0$ as $k_2 \geq k_1^2$.
Therefore, for large $t$, we have
\begin{align*}
	\text{Corr}\big(\mathcal{Z}^{\alpha}_f(s), \mathcal{Z}^{\alpha}_f(t)\big) & \sim \frac{l_2 \mathbb{E}(D_f^\alpha(s)) + d\mathbb{E}(D_f^{2\alpha}(s)) - l_1^2k_1\mathbb{E}(D_f^\alpha(s))\rho s t^{\rho-1}}{\sqrt{{ d_1t^{2\rho}\text{Var}}\big(\mathcal{Z}^{\alpha}_f(s)\big)}}\\ 
	&\sim \bigg(\frac{l_2 \mathbb{E}(D_f^\alpha(s)) + d\mathbb{E}(D_f^{2\alpha}(s))}{\sqrt{d_1{\text{Var}}\big(\mathcal{Z}^{\alpha}_f(s)\big) }}\bigg)t^{-\rho}.
\end{align*}
As $0 < \rho < 1$, the proof follows.
\end{proof}
\begin{remark}
	Similarly, it can be shown that the TCGSP-I has the LRD property.
\end{remark}

In the next result, we obtain a version of the law of iterated logarithm for TCGFSP-I. 
\begin{theorem}
	Let $\{D_{f}(t)\}_{t\ge0}$ be a L\'evy subordinator whose  associated Bern\v stein function $f$ is such that $\lim_{x\rightarrow 0+}f(\lambda x)/f(x)=\lambda^\theta$, $\lambda>0$, that is, $f$   is regularly varying at $0+$ with index $0<\theta<1$. Also, let \begin{equation*}
		g(t)=\frac{\log\log t}{\phi(t^{-1}\log\log t)},\ \ t>e,
	\end{equation*}
	where $\phi$ is the inverse of $f$. Then,\begin{equation}\label{lil}
		\liminf_{t\rightarrow\infty}\frac{\mathcal{Z}_{f}^{\alpha}(t)}{(g(t))^{\alpha}}\stackrel{d}{=}\sum_{j=1}^{k}j(\lambda_{j}-\mu_{j}) Y_{\alpha}(1)\theta^\alpha\Big(1-\theta\Big)^{\alpha(1-\theta)/\theta}.
	\end{equation}	
\end{theorem}
\begin{proof}
Using \eqref{selfsimi} in (\ref{qws11ww1}), we get
	\begin{equation}\label{self}
		\mathcal{Z}_{f}^{\alpha}(t)=\mathcal{S}(Y_{\alpha}(D_{f}(t)))\stackrel{d}{=}\mathcal{S} (D^\alpha_f(t) Y_{\alpha}(1)).
	\end{equation}  
Hence,
	\begin{align*}
		\liminf_{t\rightarrow\infty}\frac{\mathcal{Z}_{f}^{\alpha}(t)}{(g(t))^{\alpha}}&\stackrel{d}{=}\liminf_{t\rightarrow\infty}\frac{\mathcal{S}(D^\alpha_f(t) Y_{\alpha}(1))}{(g(t))^{\alpha}}\\
		&=\liminf_{t\rightarrow\infty}\Big(\frac{\mathcal{S}(D^\alpha_f(t) Y_{\alpha}(1))}{D^\alpha_f(t) Y_{\alpha}(1)}\Big)\frac{D^\alpha_f(t) Y_{\alpha}(1)}{(g(t))^{\alpha}}\\
		&\stackrel{d}{=}\sum_{j=1}^{k}j(\lambda_{j}-\mu_{j}) Y_{\alpha}(1)\Big(\liminf_{t\rightarrow\infty}\frac{D_{f}(t)}{g(t)}\Big)^{\alpha},
	\end{align*}
	where in the last step we have used the fact that $D_f(t)\to\infty$ as $t\to\infty$, a.s. (see Bertoin (1996), p. 73) and $\eqref{limitS}.$ Thus, the proof is completed on using the following law of iterated logarithm of L\'evy subordinator (see Bertoin (1996), Theorem 14, p. 92):
	\begin{equation*}\label{LIL}
		\liminf_{t\to\infty}\frac{D_{f}(t)}{g(t)}=\theta(1-\theta)^{(1-\theta)/\theta},\ \ \text{a.s.}
	\end{equation*}
\end{proof}
\begin{remark}
	On taking $\lambda_{j}=\lambda$ and $\mu_{j}=\mu$,  $j=1,2,\dots,k$  in (\ref{lil}), we get the law of iterated logarithm for a time-changed variant of the fractional Skellam process of order $k$ (see Kataria and Khandakar (2024), Theorem 4).
	
Also, on taking $\alpha=1$, $\lambda_{j}=\lambda$ and $\mu_{j}\to 0$,  $j=1,2,\dots,k$  in (\ref{lil}), we get the law of iterated logarithm for a time-changed variant of the Poisson process of order $k$ (see Sengar {\it et al.} (2020), Theorem 3.4).
\end{remark}
\begin{remark}
	
From \eqref{self}, observe that
	\begin{align*}
		\lim_{t \to \infty} \frac{\mathcal{Z}^\alpha_f(t)}{t^\alpha}&\overset{d}{=} Y_\alpha(1) \lim_{t \to \infty} \frac{\mathcal{S} \big( D^\alpha_f(t) Y_\alpha(1) \big)}{D^\alpha_f(t) Y_\alpha(1)}  \frac{D^\alpha_f(t)}{t^\alpha}\\
		&\overset{d}{=} \sum_{j=1}^{k} j (\lambda_j-\mu_j) Y_\alpha(1) \lim_{t \to \infty}  \frac{D^\alpha_f(t)}{t^\alpha} , \ \  (\text{using \eqref{limitS}})\\
		&\overset{d}{=} \sum_{j=1}^{k} j (\lambda_j-\mu_j) Y_\alpha(1) \big( \mathbb{E} \big( D_f(1) \big) \big)^\alpha,
	\end{align*}
	where the last step follows from the strong law of large numbers of a L\'evy subordinator (see Bertoin (1996), p. 92).
	
	It is known that $ Y_\alpha(1) $ is not infinitely divisible (see Vellaisamy and Kumar (2018)). Thus, using a contradiction argument, it can be shown that the one-dimensional distributions of TCGFSP-I are not infinitely divisible.
	
\end{remark}
\subsection{Some special cases of the TCGFSP-I}
In this subsection, we discuss three special cases of TCGFSP-I  and TCGSP-I by taking three specific L\'evy subordinators, namely, the gamma subordinator, the tempered stable subordinator (TSS) and the inverse Gaussian subordinator (IGS) as a time-change component in the GFSP and GSP.
\subsubsection{GFSP time-changed by gamma subordinator} 
Let $\{Z(t)\}_{t\ge0}$ be a gamma subordinator whose associated   Bern\v stein function is (see Applebaum (2009), p. 55)
\begin{equation*}
	f_{1}(s)=b\log(1+s/a), \ \ a>0, \  b>0,  \ s>0.
\end{equation*}
 Its pdf $g(x,t)$ is given by
\begin{equation*}
	g(x,t)=\frac{a^{bt}}{\Gamma(bt)}x^{bt-1}e^{-ax},\  x>0,
\end{equation*}
which satisfies the following fractional differential equation (see Beghin and Vellaisamy (2018), Lemma 2.2):
\begin{equation}\label{diffgam}
\mathbb{D}_t^{\gamma}g(x,t)=b \mathbb{D}_t^{\gamma-1}\big(\log(a x)-\kappa(bt)\big)g(x,t), \ 
	g(x,0)=0.  
\end{equation}
 Here, $\gamma\ge1$,  $\kappa(x)\coloneqq\Gamma^{\prime}(x)/\Gamma(x)$ is the digamma function and $\mathbb{D}_t^{\gamma}$ is the Riemann-Liouville fractional derivative defined in (\ref{RLd}).	

On taking $f_1$ as the  Bern\v stein function in \eqref{qws11ww1}, we get the  GFSP time-changed by an independent gamma subordinator as 
\begin{equation}\label{fgam}
	\mathcal{Z}^\alpha_{f_{1}}(t)\coloneqq \mathcal{S}^\alpha
	(Z(t)),\ t\ge0.
\end{equation}

\begin{proposition}
	Let $\gamma\ge1$ and $\kappa(x)$ be the digamma function. Then, the pmf $p^\alpha_{f_{1}}(n,t)=\mathrm{Pr}\{\mathcal{Z}^\alpha_{f_{1}}(t)=n\}$, $n\in\mathbb{Z}$ of $\{\mathcal{Z}^\alpha_{f_{1}}(t)\}_{t \ge 0}$ solves the following differential equation:
	\begin{equation*}
		\mathbb{D}_t^{\gamma}p^\alpha_{f_{1}}(n,t)=b \mathbb{D}_t^{\gamma-1}\big(\log (a)-\kappa(bt)\big)p^\alpha_{f_{1}}(n,t)+b\int_{0}^{\infty}p^\alpha(n,x)\log (x) \mathbb{D}_t^{\gamma-1}g(x,t)\,\mathrm{d}x,
	\end{equation*}
	where $p^\alpha(n,x)=\mathrm{Pr}\{\mathcal{S}^\alpha(x)=n\}$ is the pmf of GFSP.
\end{proposition}
\begin{proof}
	From (\ref{fgam}), we have
	\begin{equation}\label{qaza122}
		p^\alpha_{f_{1}}(n,t)=\int_{0}^{\infty}p^\alpha(n,x)g(x,t)\,\mathrm{d}x.
	\end{equation}
	Taking the Riemann-Liouville fractional derivative in (\ref{qaza122}) and using \eqref{diffgam}, we get
	\begin{align*}
		\mathbb{D}_t^{\gamma}p^\alpha_{f_{1}}(n,t)&=\int_{0}^{\infty}p^\alpha(n,x)\mathbb{D}_t^{\gamma}g(x,t)\,\mathrm{d}x\\
		&=b\int_{0}^{\infty}p^\alpha(n,x)\mathbb{D}_t^{\gamma-1}\big(\log(a x)-\kappa(bt)\big)g(x,t)\,\mathrm{d}x \\
		&=b \mathbb{D}_t^{\gamma-1}\log (a) \int_{0}^{\infty}p^\alpha(n,x)g(x,t)\,\mathrm{d}x+b \int_{0}^{\infty} p^\alpha(n,x)\log (x)\mathbb{D}_t^{\gamma-1}g(x,t)\,\mathrm{d}x\\
		&\ \ - b \mathbb{D}_t^{\gamma-1}\kappa(bt)\int_{0}^{\infty}p^\alpha(n,x)g(x,t)\,\mathrm{d}x.
	\end{align*}
	The proof is completed on using (\ref{qaza122}).
\end{proof}

On using \eqref{meantcgfsp}-\eqref{vartcgfsp}, the mean, variance and covariance of $\{\mathcal{Z}^\alpha_{f_{1}}(t)\}_{t \ge 0}$ can be written as follows:
\begin{align}
	\mathbb{E}( \mathcal{Z}_{f_1}^{\alpha}(t) ) & = l_1 \mathbb{E}( Z^{\alpha}(t) ),\nonumber\\
	\text{Var}( \mathcal{Z}_{f_1}^{\alpha}(t) )&=  \mathbb{E}( Z^{\alpha}(t) ) ( l_2 - l_1^2\mathbb{E}( Z^{\alpha}(t) ) ) + 2d \mathbb{E}( Z^{2\alpha}(t) ),\label{varfgam}\\
	\text{Cov}(\mathcal{Z}_{f_1}^{\alpha}(s), \mathcal{Z}_{f_1}^{\alpha}(t))&= l_2\mathbb{E}( Z^{\alpha}(s) ) + d \mathbb{E}( Z^{2\alpha}(s) ) - l_1^2 \mathbb{E}( Z^{\alpha}(s) ) \mathbb{E}( Z^{\alpha}(t) ) \nonumber \\
	&\hspace{1cm}+  l_1^2 \alpha \mathbb{E}\big( Z^{2\alpha}(t) B(\alpha, \alpha +1; {Z(s)}/{Z(t)} )\big)\label{covfgam}.
\end{align}
Note that moments of the gamma subordinator
has the following asymptotic expansion as $t\to \infty$ (see Maheshwari and Vellaisamy (2019), Remark 3.3):
\begin{equation}\label{stir}
	\mathbb{E}(Z^\alpha(t))\sim({bt}/{a})^\alpha,\ \ \mathbb{E}(Z^{2\alpha}(t))\sim({bt}/{a})^{2\alpha}.
\end{equation}
Thus, from Theorem \eqref{lrd}, it follows that the process $ \{\mathcal{Z}^\alpha_{f_1}(t)\}_{t \geq 0} $ exhibits the LRD property.

For a fixed $ h > 0 $, the increment process of  $\{\mathcal{Z}_{f_1}^{\alpha}(t)\}_{t\ge0}$ is defined as
\begin{equation*}
	\mathcal{S}_h^\alpha(t)\coloneqq\mathcal{Z}_{f_1}^\alpha(t+h)-\mathcal{Z}_{f_1}^\alpha(t) =\mathcal{S}^\alpha(Z(t + h)) - \mathcal{S}^\alpha(Z(t)), \ \ t \geq 0.
\end{equation*}
Next, we show that the increment process $ \{ \mathcal{S}_h^\alpha(t) \}_{t \geq 0} $ exhibits the SRD property. For this purpose the following asymptotic  result of gamma subordinator will be used (see Maheshwari and Vellaisamy (2016), Lemma 2): Let $h>0$ be fixed and $0< \alpha <1$. Then for large $t$, we have
\begin{equation}\label{gamaasym3}
	\mathbb{E}\big(Z^{2\alpha}(t+h)B(\alpha,1+\alpha;Z(t)/Z(t+h))\big) \sim B(\alpha,1+\alpha)\mathbb{E}(Z^{2\alpha}(t+h)).
\end{equation}

\begin{theorem}
	The process $ \{ \mathcal{S}_h^\alpha(t) \}_{t \geq 0} $ has the SRD property.
\end{theorem}
\begin{proof}
	 Let $s>0$ be fixed such that $0 <s+h \le t$. Then,
\begin{align}\label{pf1}
	\text{Cov}\big(\mathcal{S}_h^{\alpha}(s), \mathcal{S}_h^{\alpha}(t)\big)&=\text{Cov}\big(\mathcal{Z}^{\alpha}_{f_1}(s+h)-\mathcal{Z}^{\alpha}_{f_1}(s),\mathcal{Z}^{\alpha}_{f_1}(t+h)-\mathcal{Z}^{\alpha}_{f_1}(t)\big)\nonumber\\&=\text{Cov}\big(\mathcal{Z}^{\alpha}_{f_1}(s+h),\mathcal{Z}^{\alpha}_{f_1}(t+h)\big)+\text{Cov}\big(\mathcal{Z}^{\alpha}_{f_1}(s),\mathcal{Z}^{\alpha}_{f_1}(t)\big)\nonumber\\&\hspace{1cm}-\text{Cov}\big(\mathcal{Z}^{\alpha}_{f_1}(s+h),\mathcal{Z}^{\alpha}_{f_1}(t)\big)-\text{Cov}\big(\mathcal{Z}^{\alpha}_{f_1}(s),\mathcal{Z}^{\alpha}_{f_1}(t+h)\big).
\end{align}
By taking gamma subordinator in \eqref{cov prs asym} and using \eqref{stir} for large $t$, we get  
\begin{equation}\label{pf2}
	\text{Cov}\big(\mathcal{Z}^{\alpha}_{f_1}(s), \mathcal{Z}^{\alpha}_{f_1}(t)\big) \sim l_2 \mathbb{E}(Z^\alpha(s)) + d\mathbb{E}(Z^{2\alpha}(s)) - l_1^2({b}/{a})^\alpha \mathbb{E}(Z^\alpha(s))\alpha s t^{\alpha-1}.
\end{equation}
Combining \eqref{pf1} and \eqref{pf2}, we deduce the following as $t \to \infty$:
\begin{align}\label{proof1}
	\text{Cov}(\mathcal{S}_h^{\alpha}(s), \mathcal{S}_h^{\alpha}(t))
	&\sim-l_1^2({b}/{a})^\alpha \alpha \big((s+h)\mathbb{E}(Z^\alpha(s+h))(t+h)^{\alpha -1}+s\mathbb{E}(Z^\alpha(s))t^{\alpha -1}\nonumber\\
	& \hspace{1cm}-(s+h)\mathbb{E}(Z^\alpha(s+h))t^{\alpha -1}-s\mathbb{E}(Z^\alpha(s))(t+h)^{\alpha -1}\big)\nonumber\\
	&= -l_1^2({b}/{a})^\alpha \alpha \big((s+h)\mathbb{E}(Z^\alpha(s+h))-s\mathbb{E}(Z^\alpha(s))\big)\big((t+h)^{\alpha -1}-t^{\alpha -1}\big)\nonumber\\
	&\sim l_1^2({b}/{a})^\alpha\alpha (1-\alpha )h\big((s+h)\mathbb{E}(Z^\alpha(s+h))-s\mathbb{E}(Z^\alpha(s))\big)t^{\alpha -2}.
\end{align}
From \eqref{varfgam} and \eqref{covfgam}, we get
\begin{align}\label{proof2}
	\text{Var}(\mathcal{S}_h^\alpha(t))
	&= \text{Var}(\mathcal{Z}^{\alpha}_{f_1}(t+h))+ \text{Var}(\mathcal{Z}^{\alpha}_{f_1}(t))-2\text{Cov}(\mathcal{Z}^{\alpha}_{f_1}(t),\mathcal{Z}^{\alpha}_{f_1}(t+h))\nonumber\\
	&=l_2 \big(\mathbb{E}( Z^{\alpha}(t+h) )-\mathbb{E}( Z^{\alpha}(t) )\big)  - l_1^2\big(\mathbb{E}( Z^{\alpha}(t+h) )-\mathbb{E}( Z^{\alpha}(t) )\big)^2 \nonumber \\
	& \hspace{1cm}+ 2d \mathbb{E}(Z^{2\alpha}(t+h) )- 2l_1^2 \alpha \mathbb{E}\big( Z^{2\alpha}(t+h) B(\alpha, 1 + \alpha; {Z(t)}/{Z(t+h)}) \big)\nonumber\\
	&\sim l_2 ({b}/{a})^\alpha ((t+h)^{\alpha }-t^\alpha )-l_1^2({b}/{a})^{2\alpha }((t+h)^{\alpha }-t^\alpha )^2, \ (\text{using} \ \eqref{stir} \  \text{and} \ \eqref{gamaasym3})\nonumber\\
	&\sim l_2({b}/{a})^\alpha \alpha  ht^{\alpha -1}-l_1^2({b}/{a})^{2\alpha }\alpha ^2h^2t^{2\alpha -2}\nonumber\\&\sim l_2({b}/{a})^\alpha \alpha  ht^{\alpha -1}.
\end{align}
From \eqref{proof1} and \eqref{proof2}, we have 
\begin{align*}
	\text{Corr}(\mathcal{S}_h^{\alpha}(s),\mathcal{S}_h^{\alpha}(t))
	&\sim\frac{l_1^2({b}/{a})^\alpha \alpha (1-\alpha )h\big((s+h)\mathbb{E}(Z^\alpha(s+h))-s\mathbb{E}(Z^\alpha(s))\big)t^{\alpha -2}}{\sqrt{\text{Var}(\mathcal{S}_h^\alpha(s)}\sqrt{l_2({b}/{a})^\alpha \alpha  ht^{\alpha -1}}}\\
	&\sim c_1(s)t^{-{(3-\alpha )}/{2}},
\end{align*}
where 
\begin{equation*}
	c_1(s)=\frac{l_1^2({b}/{a})^\alpha \alpha (1-\alpha )h((s+h)\mathbb{E}(Z^\alpha(s+h))-s\mathbb{E}(Z^\alpha(s)))}{\sqrt{\text{Var}(\mathcal{S}_h^\alpha(s))}\sqrt{l_2({b}/{a})^\alpha \alpha  h}}.
\end{equation*} 
Thus, the process $ \{ \mathcal{S}_h^\alpha(t) \}_{t \geq 0} $ exhibits the SRD property as $1< (3-\alpha)/2<3/2$.
\end{proof}
For the Bern\v stein function $f_1$, the TCGSP-I $\{\mathcal{Z}_{f}(t)\}_{t\ge0}$ reduces to the GSP time-changed by an independent gamma subordinator. That is,
\begin{equation}\label{gam}
	\mathcal{Z}_{f_{1}}(t)\coloneqq \mathcal{S}
	(Z(t)).
\end{equation}
Next, we obtain a system of differential equations that governs the state probabilities $p_{f_{1}}(n,t)=\mathrm{Pr}\{\mathcal{Z}_{f_{1}}(t)=n\}$, $n\in\mathbb{Z}$ of $\{\mathcal{Z}_{f_{1}}(t)\}_{t \ge 0}$. For this purpose, the following differential equation for the pdf of gamma subordinator will be used (see Beghin (2015b), Eq. 9):
\begin{equation}\label{res}
	\left\{
	\begin{array}{l}
		\dfrac{\partial }{\partial x}g(x,t)=-b(1-e^{-\partial _{t}/a})g(x,t),
		\\
		g(x,0)=\delta (x), \\
		\lim_{|x|\rightarrow +\infty }g(x,t)=0,
	\end{array}
	\right.  
\end{equation}%
where $x\ge0$, $t\geq 0$, $\partial_{t}=\frac{\partial}{\partial t}$, $\delta (x)$ is the Dirac delta function and $e^{-\partial _{t}/a}$ is a  shift operator which is defined for any analytic function $f:\mathbb{R}\to \mathbb{R}$ as follows:
\begin{equation*}\label{78}
	e^{c\partial_{t}}f(t)\coloneqq \sum_{n=0}^{\infty}\frac{(c\partial_{t})^{n}}{n!}f(t)=f(t+c), \ \  c \in \mathbb{R} .
\end{equation*}
\begin{proposition}
	The pmf $p_{f_1}(n,t)$, $n\in\mathbb{Z}$ satisfies the following differential equation:
	\begin{equation*}
		b\big(e^{-\partial_t / a}-1\big) p_{f_1}(n,t) = \Lambda \big(p_{f_1}(n-1,t)-p_{f_1}(n,t)\big) +\bar{\Lambda} \big(p_{f_1}(n,t)-p_{f_1}(n+1,t)\big),
	\end{equation*}
	with initial conditions
	\begin{equation*}
		p_{f_1}(n,0) =
		\begin{cases}
			1, & n = 0, \\
			0, & n \neq 0.
		\end{cases}
	\end{equation*}
	\begin{proof}
		
 From \eqref{gam}, we can write
	\begin{equation}\label{45}
		p_{f_1}(n,t) = \int_0^{\infty} p(n,x) g(x,t)\mathrm{d}x, 
	\end{equation}
	where $p(n,x)$ is the pmf of GSP given in \eqref{pmfgsp}.
	
	We apply the shift operator $e^{-\partial_t / a}$ to both sides of \eqref{45} and use  \eqref{res} to get
{\small	\begin{align*}
		e^{-\partial_t / a} p_{f_1}(n,t) &= \int_0^{\infty} p(n,x) \Big( g(x,t) + \frac{1}{b} \frac{\partial}{\partial x} g(x,t) \Big)\mathrm{d}x \\
		&= p_{f_1}(n,t) + \frac{1}{b} \Big( p(n,x) g(x,t) \Big|_{x=0}^{\infty} - \int_0^{\infty} g(x,t) \frac{\mathrm{d}}{\mathrm{d}x} p(n,x)\mathrm{d}x \Big) \\
		&= p_{f_1}(n,t) - \frac{1}{b} \int_0^{\infty} g(x,t) \Big( \Lambda \big(p(n-1,x)-p(n,x)\big)-\bar{\Lambda}\big(p(n,x)-p(n+1,x)\big)\Big)\mathrm{d}x,
	\end{align*}}
	where in the last step we have used \eqref{digsp}. Now, the required result follows on using \eqref{45}.
		\end{proof}
\end{proposition}
\subsubsection{GSP time-changed by tempered stable subordinator}
 Let $f_{2}(s)$ be the Bern\v stein function associated with TSS $\{\mathscr{D}_{\eta,\theta}(t)\}_{t\ge0}$ with stability index  $0<\theta<1$   and  tempering parameter $\eta>0$. Then
\begin{equation}\label{bs1}
	f_{2}(s)=(\eta+s)^{\theta}-\eta^{\theta},\ \ s>0.
\end{equation}
Let $\delta_0(x)$ denote the Dirac delta function. The pdf $ h_{\eta, \theta}(x, t) $ of TSS
 solves the following (see Beghin (2015a), Eq. (15)):
\begin{equation}\label{difftss1}
	\frac{\partial}{\partial x}h_{\eta,\theta}(x,t)=-\eta h_{\eta,\theta}(x,t)+\Big(\eta^{\theta}-\frac{\partial}{\partial t}\Big)^{1/\theta}h_{\eta,\theta}(x,t),
\end{equation}
with the initial conditions $h_{\eta,\theta}(x,0)=\delta_0(x)$ and $h_{\eta,\theta}(0,t)=0$.

We define the GSP time-changed by an independent TSS  by taking $f_{2}$ as the Bern\v stein function in \eqref{zft} as follows:
\begin{equation}\label{tss}
	\mathcal{Z}_{f_{2}}(t)\coloneqq \mathcal{S}(\mathscr{D}_{\eta,\theta}(t)),\ \ t\ge0.
\end{equation}	

\begin{proposition}
	The pmf $p_{f_{2}}(n,t)=\mathrm{Pr}\{\mathcal{Z}_{f_{2}}(t)=n\}$ of  $\{\mathcal{Z}_{f_{2}}(t)\}_{t\ge0}$ satisfies the following system of differential equations:
	\begin{equation*}
		\Big(\eta^{\theta}-\frac{\mathrm{d}}{\mathrm{d} t}\Big)^{1/\theta}p_{f_{2}}(n,t)=(\eta +\Lambda+\bar{\Lambda}) p_{f_{2}}(n,t)-\Lambda p_{f_{2}}(n-1,t)-\bar{\Lambda} p_{f_{2}}(n+1,t),\ \ n\in \mathbb{Z}.
	\end{equation*}
\end{proposition}
\begin{proof}

	From (\ref{tss}), we have
	\begin{equation}\label{wslp09}
		p_{f_{2}}(n,t)=\int_{0}^{\infty}p(n,x)	h_{\eta,\theta}(x,t)\,\mathrm{d}x.
	\end{equation}
	On using \eqref{difftss1} in (\ref{wslp09}), we get 
	\begin{align*}
		\big(\eta^{\theta}-\frac{\mathrm{d}}{\mathrm{d} t}\big)^{1/\theta}p_{f_{2}}(n,t)&=\int_{0}^{\infty}p(n,x)\Big(\eta h_{\eta,\theta}(x,t)+\frac{\partial}{\partial x}	h_{\eta,\theta}(x,t)\Big)\,\mathrm{d}x \\
		&=\eta p_{f_{2}}(n,t)-\int_{0}^{\infty}h_{\eta,\theta}(x,t)\frac{\mathrm{d}}{\mathrm{d} x}p(n,x)	\,\mathrm{d}x\\
		&=\eta p_{f_{2}}(n,t)- \int_{0}^{\infty}\Big(\Lambda\big(p(n-1,x)-p(n,x)\big)\\
		&\hspace*{1.5cm}-\bar{\Lambda}\big(p(n,x)-p(n+1,x)\big)\Big)h_{\eta,\theta}(x,t)\,\mathrm{d}x,
	\end{align*}
	where in the last step we have used \eqref{digsp} and in the penultimate step we have used $\lim_{x\to 0}h_{\eta,\theta}(x,t)=\lim_{x\to \infty}h_{\eta,\theta}(x,t)=0$ (see Gupta {\it et al.} (2020b), Eq. (3.50)). Finally, the proof concludes in view of \eqref{wslp09}.
\end{proof}
In particular, the pmf $p_{f_{2}}(n,t)$ satisfies the following if $\theta^{-1}=m \ge 2$ is an integer:
\begin{equation*}
	\sum_{k=1}^m (-1)^k \binom{m}{k}\eta^{(1-k/m)}\frac{\mathrm{d}^k}{\mathrm{d} t^k}p_{f_{2}}(n,t)=(\Lambda+\bar{\Lambda}) p_{f_{2}}(n,t)-\Lambda p_{f_{2}}(n-1,t)-\bar{\Lambda} p_{f_{2}}(n+1,t).
\end{equation*}

\subsubsection{GSP time-changed by inverse Gaussian subordinator}
Let $\{Y(t)\}_{t\ge0}$ be an IGS with parameters $\delta>0$ and $\gamma>0$. Its associated Bern\v stein function is (see Applebaum (2009), Eq. (1.26))
\begin{equation}\label{ploiuy67}
	f_{3}(s)=\delta\big(\sqrt{2s+\gamma^2}-\gamma\big), \ \ s>0.
\end{equation}
Its pdf $q(x,t)$ is given by (see Applebaum (2009), Eq. (1.27))
\begin{equation*}
	q(x,t)=(2\pi)^{-1/2}\delta tx^{-3/2}\exp\big(\delta\gamma t-\frac{1}{2}(\delta^{2}t^{2}x^{-1}+\gamma^{2}x)\big), \ \ x>0,
\end{equation*}
which solves (see Vellaisamy and Kumar (2018), Eq. (3.3)) 
\begin{equation}\label{diffigs}
	\Big(\frac{\partial^2}{\partial t^2}-2\delta \gamma \frac{\partial}{\partial t}\Big) q(x,t)= 2\delta^2\frac{\partial}{\partial x}q(x,t).
\end{equation}
The GSP time-changed by an independent IGS is defined as
\begin{equation}\label{ig}
	\mathcal{Z}_{f_{3}}(t)\coloneqq \mathcal{S}(Y(t)),\ \ t\ge0.
\end{equation}

\begin{proposition}\label{p4.1}
	The pmf $p_{f_{3}}(n,t)=\mathrm{Pr}\{\mathcal{Z}_{f_{3}}(t)=n\}$, $n\in \mathbb{Z} $ of $\{\mathcal{Z}_{f_{3}(t)}\}_{t\ge 0}$ solves 
	\begin{equation*}
		\Big(\frac{\mathrm{d}^2}{\mathrm{d}t^2}-2\delta \gamma \frac{\mathrm{d}}{\mathrm{d}t}\Big) p_{f_{3}}(n,t)=-2\delta^2\Big(\Lambda\big(p_{f_{3}}(n-1,t)-p_{f_{3}}(n,t)\big)-\bar{\Lambda}\big(p_{f_{3}}(n,t)-p_{f_{3}}(n+1,t)\big)\Big).
	\end{equation*}
\end{proposition}
\begin{proof}
	From (\ref{ig}), we have
	\begin{equation}\label{123}
		p_{f_{3}}(n,t)=\int_0^{\infty}p(n,x)q(x,t)\,\mathrm{d}x.
	\end{equation}	
By using \eqref{diffigs} in (\ref{123}), we get
	\begin{align*}  
		\Big(\frac{\mathrm{d}^2}{\mathrm{d} t^2}-2\delta \gamma \frac{\mathrm{d}}{\mathrm{d} t}\Big) p_{f_{3}}(n,t)
		&=\int_0^{\infty}p(n,x) \Big(\frac{\partial^2}{\partial t^2}-2\delta \gamma \frac{\partial}{\partial t}\Big) q(x,t)\,\mathrm{d}x \\
		&= 2\delta^2 \int_0^{\infty} p(n,x)\frac{\partial}{\partial x} q(x,t)\,\mathrm{d}x  \\ 
		&= -2\delta^2 \int_0^{\infty} q(x,t)\frac{\mathrm{d}}{\mathrm{d}x} p(n,x)\, \mathrm{d}x \\
		&=-2\delta^2\int_0^{\infty}\Big(\Lambda\big(p(n-1,x)-p(n,x)\big)  \\
		&\hspace*{1.5cm}-\bar{\Lambda}\big(p(n,x)-p(n+1,x)\big)\Big)q(x,t)\,\mathrm{d}x,
	\end{align*}
	where in the last step we have used \eqref{digsp} and in the penultimate step we have used $\lim_{x\to\infty}q(x,t)=\lim_{x\to 0}q(x,t)=0$. Now, the proof is completed on using \eqref{123}.
\end{proof}

\section{GFSP time-changed by inverse subordinator}\label{4}

In this section, we introduce another time-changed version of the GFSP  by using the inverse subordinator. Recall that the first passage time  of a L\'evy subordinator $\{D_f (t)\}_{t\ge0}$ is called the inverse subordinator $\{H_f (t)\}_{t\ge0}$. 
We note that $\mathbb{E}\big(H^r_f(t)\big)<\infty$ for all $r>0$ (see Aletti {\it et al.} (2018), Section 2.1). 

Let 
\begin{equation*}
	\bar{\mathcal{Z}}^{\alpha}_{f}(t)\coloneqq \mathcal{S}^{\alpha}(H_f (t)),\ \ t\ge0,
\end{equation*}
be the GFSP time changed by an independent inverse subordinator. We call the process $\{\bar{\mathcal{Z}}^{\alpha}_{f}(t)\}_{t\ge 0}$ the time-changed generalized fractional Skellam process II (TCGFSP-II). 
For $\alpha=1$, the TCGFSP-II reduces to the time-changed generalized Skellam process-II (TCGSP-II), that is, 
\begin{equation}\label{loppp1}
	\bar{\mathcal{Z}}_{f}(t)\coloneqq \mathcal{S}(H_f (t)),\ \ t\ge0.
\end{equation}
Next, we discuss some distributional properties of the TCGFSP-II and TCGSP-II.

The pgf of TCGFSP-II is given by
\begin{equation*}
\mathbb{E}\big(u^{\bar{\mathcal{Z}}_f^\alpha(t)}\big)=\sum_{n=0}^{\infty}\frac{\big(\sum_{j=1}^{k}\left(\lambda_{j}(u^{j}-1)+\mu_{j}(u^{-j}-1)\right)\big)^n}{\Gamma(n\alpha +1 )}\mathbb{E}\big((H_f(t))^{n\alpha }\big),\ \ |u|\le1,
\end{equation*}
whose proof follows similar lines to that of Theorem \ref{thpgfI}.

Let us assume that $0 < s \leq t < \infty$, and recall that $m_1=\sum_{j=1}^{k} j(\lambda_j-\mu_j)$, $m_2=\sum_{j=1}^{k} j^2(\lambda_j+\mu_j)$, $l_1 =\frac{m_1}{\Gamma(1 + \alpha)}$, $l_2 =\frac{m_2}{\Gamma(1 + \alpha)}$ and $d = \alpha l_1^2 B(\alpha, 1 + \alpha)$. Then,
\begin{align*}
	&(i)\ \mathbb{E}\big(\bar{\mathcal{Z}} _f^{\alpha}(t) \big) = l_1 \mathbb{E}\big( H_f^{\alpha}(t) \big),\\
	&(ii)\ \text{Var}\big( \bar{\mathcal{Z}}_f^{\alpha}(t) \big) =  \mathbb{E}\big( H_f^{\alpha}(t) \big) \big( l_2 - l_1^2\mathbb{E}\big( H_f^{\alpha}(t) \big) \big) + 2d \mathbb{E}\big( H_f^{2\alpha}(t) \big),\\
	&(iii)\ \text{Cov}\big(  \bar{\mathcal{Z}}_f^{\alpha}(s), \bar{\mathcal{Z}}_f^{\alpha}(t) \big) = l_2\mathbb{E}\big( H_f^{\alpha}(s) \big) + d \mathbb{E}\big( H_f^{2\alpha}(s) \big) - l_1^2 \mathbb{E}\big( H_f^{\alpha}(s) \big) \mathbb{E}\big( H_f^{\alpha}(t) \big) \\ &\ \  \hspace{8cm} + l_1^2 \alpha \mathbb{E}\big( H_f^{2\alpha}(t) B(\alpha, \alpha+1; {H_f(s)}/{H_f(t)}) \big).
\end{align*}
The proof of $(i)-(iii)$ follows similar lines to the corresponding results of TCGFSP-I (see Section \ref{3}). Thus, the proofs are omitted.

The pmf $\bar{p}_{f}(n,t)=\mathrm{Pr}\{\bar{\mathcal{Z}}_{f}(t)=n\}$ of TCGSP-II is given by 
\begin{equation*}
	\bar{p}_{f}(n,t)=\sum_{m=\max(0,-n)}^{\infty} \frac{\Lambda^{m+n} \bar{\Lambda}^{m} }{(m+n)! m!} \mathbb{E}\big(e^{-(\Lambda+\bar{\Lambda})H_{f}(t)} H_{f}^{2m+n}(t)\big), \ \ n \in \mathbb{Z}.
\end{equation*}
The proof of the above result follows similar lines to that of Theorem \ref{thpmfI}.
\begin{theorem}
	Let Condition \textbf{I} hold. Then, the marginal distributions of TCGSP-II satisfy
	\begin{equation}\label{gcpd}
	{}^{f}\mathscr{D}_t \bar{p}_f(n,t)=\Lambda\big(\bar{p}_f(n-1,t)-\bar{p}_f(n,t)\big)-\bar{\Lambda}\big(\bar{p}_f(n,t)-\bar{p}_f(n+1,t)\big), \ \ n \in \mathbb{Z},
\end{equation}
	with initial conditions 
	$$\bar{p}_f(n,0) =
	\begin{cases} 
		1, & n = 0, \\
		0, & n \neq 0,
	\end{cases}$$
	where ${}^{f}\mathscr{D}_t$ is the generalized Caputo derivative with respect to the Bern\v{s}tein function
	$f$ defined in \eqref{gcaputo}.
\end{theorem}
\begin{proof}
	From \eqref{loppp1}, we have
	\begin{equation}\label{4.4}
		\bar{p}_{f}(n,t) = \int_0^{\infty} p(n,x) l_f(x,t)\mathrm{d}x, 
	\end{equation}
	where $l_f(x,t)$ is the density of inverse subordinator $\{H_f(t)\}_{t \ge0}$. 
	
	On taking the generalized Riemann-Liouville derivative ${}^{f}\mathbb{D}_t$  and using 
	\eqref{2.28} in \eqref{4.4}, we get
	\begin{align}\label{33}
		{}^{f}\mathbb{D}_t \bar{p}_f(n,t) &= 
		-\int_0^\infty p(n,x) \frac{\partial}{\partial x} l_f(x,t)\mathrm{d}x \nonumber\\
		&= - \big[ p(n,x) l_f(x,t)\big]_{x=0}^{x=\infty}+ \int_0^\infty  l_f(x,t)\frac{\mathrm{d}}{\mathrm{d} x}  p(n,x) \, \mathrm{d}x\nonumber \\
		&= p(n,0) \nu(t)+\int_0^\infty l_f(x,t)\big(\Lambda(p(n-1,x)-p(n,x))- \nonumber\\& \hspace{5cm}\ \ \bar{\Lambda}(p(n,x)-p(n+1,x))\big) \mathrm{d}x , \ \  (\text{using} \ \  \eqref{digsp})\nonumber\\
		&=p(n,0) \nu(t) + \Lambda\big(\bar{p}_f(n-1,t)-\bar{p}_f(n,t)\big)-\bar{\Lambda}\big(\bar{p}_f(n,t)-\bar{p}_f(n+1,t)\big).
	\end{align}
	From \eqref{2.27}, we get
	\begin{equation}\label{000}
		{}^{f}\mathscr{D}_t \bar{p}_f(n,t) = {}^{f}\mathbb{D}_t \bar{p}_f(n,t) - \bar{p}_f(n,0)\nu(t), 
	\end{equation}
	where
	$$
	\bar{p}_f(n,0) = \int_0^\infty p(n,x) l_f(x, 0)\mathrm{d}x = \int_0^\infty p(n,x) \delta(x)\mathrm{d}x= p(n,0).
	$$
	Finally, we get the required result on using \eqref{000} in \eqref{33}.
\end{proof}
\begin{remark}
For $f(s)=s^\alpha$, $ 0<\alpha <1$, the generalized Caputo derivative reduces to (see Toaldo (2015), Remark 2.6)
	\begin{equation*}
		{}^{f}\mathscr{D}_{t} w(t)=\frac{\mathrm{d}^\alpha}{\mathrm{d} t^\alpha}w(t),
	\end{equation*}
	where $\dfrac{\mathrm{d}^\alpha}{\mathrm{d} t^\alpha}$ is the Caputo fractional derivative. Hence, the differential equation \eqref{gcpd} reduces to the governing differential equation for the state probabilities of the GFSP (see Kataria and Khandakar (2022b), Eq. (3.8)).
\end{remark}
Next, we discuss two particular cases of the TCGSP-II.
\subsection{GSP time-changed by the inverse TSS}
Recall that  $\{\mathscr{D}_{\eta,\theta}(t)\}_{t\ge0}$  is a TSS  with stability index  $0<\theta<1$   and  tempering parameter $\eta>0$. Its first passage time, that is,  
\begin{equation*}
	\mathscr{L}_{\eta,\theta}(t)\coloneqq \inf \{ r\ge 0 : \mathscr{D}_{\eta,\theta}(r)>t\},\ \ t\ge0,
\end{equation*}
is called the inverse TSS. 
 
The following result will be used for the density function $ l_{\eta,\theta}(x,t)$ of $\{\mathscr{L}_{\eta,\theta}(t)\}_{t\ge0}$ (see Kumar {\it et al.} (2019), Eq. (25)):
\begin{equation}\label{qmk11}
	\frac{\partial}{\partial x}l_{\eta,\theta}(x,t)=-\Big(\eta+\frac{\partial}{\partial t}\Big)^{\theta}l_{\eta,\theta}(x,t)+\eta^{\theta} l_{\eta,\theta}(x,t)-t^{-\theta}E^{1-\theta}_{1,1-\theta}(-\eta t)\delta_0(x),
\end{equation}
where $\delta_0(x)=l_{\eta,\theta}(x,0)$ and $E^{1-\theta}_{1,1-\theta}(\cdot)$ is the three-parameter Mittag-Leffler function defined in \eqref{mitag}.

We define the following subordinated process:
\begin{equation}\label{inv1}
	\bar{\mathcal{Z}}_{f_{2}}(t)\coloneqq \mathcal{S}(\mathscr{L}_{\eta,\theta} (t)),\ \ t\ge0,
\end{equation}
where the GSP is independent of $\{\mathscr{L}_{\eta,\theta}(t)\}_{t\ge0}$ and  $f_{2}$ is given in (\ref{bs1}).
\begin{proposition}
	The pmf $\bar{p}_{f_{2}}(n,t)=\mathrm{Pr}\{\bar{\mathcal{Z}}_{f_2}(t)=n\}$, $n\in \mathbb{Z}$ solves the following differential equation:
	\begin{align*}
		\Big(\eta+\frac{\mathrm{d}}{\mathrm{d} t}\Big)^{\theta}\bar{p}_{f_{2}}(n,t)&=(\eta^{\theta}-\Lambda-\bar{\Lambda})\bar{p}_{f_{2}}(n,t)+\Lambda\bar{p}_{f_{2}}(n-1,t)+\bar{\Lambda}\bar{p}_{f_{2}}(n+1,t)\\
		&\ \ \  -t^{-\theta}E^{1-\theta}_{1,1-\theta}(-\eta t)p(n,0)+p(n,x)l_{\eta,\theta}(x,t)\big|_{x=0}.
	\end{align*}
\end{proposition}
\begin{proof}
	From (\ref{inv1}), we have
	\begin{equation}\label{gfdt}
		\bar{p}_{f_{2}}(n,t)=\int_{0}^{\infty}p(n,x)l_{\eta,\theta}(x,t)\,\mathrm{d}x.
	\end{equation}
	From (\ref{qmk11}) and (\ref{gfdt}), we get
	\begin{align*}
		\Big(\eta+\frac{\mathrm{d}}{\mathrm{d} t}\Big)^{\theta}\bar{p}_{f_{2}}(n,t)&=\int_{0}^{\infty}p(n,x)\Big(\eta^{\theta} l_{\eta,\theta}(x,t)-t^{-\theta}E^{1-\theta}_{1,1-\theta}(-\eta t)\delta_0(x)-\frac{\partial}{\partial x}l_{\eta,\theta}(x,t)\Big)\,\mathrm{d}x \\
		&=\eta^{\theta}\bar{p}_{f_{2}}(n,t)-t^{-\theta}E^{1-\theta}_{1,1-\theta}(-\eta t)\int_{0}^{\infty}p(n,x)\delta_0(x)\,\mathrm{d}x  \\
		&\ \ \ +p(n,x)l_{\eta,\theta}(x,t)\big|_{x=0}+\int_{0}^{\infty}l_{\eta,\theta}(x,t)\frac{\mathrm{d}}{\mathrm{d}x}p(n,x)\,\mathrm{d}x  \\
		&=\eta^{\theta}\bar{p}_{f_{2}}(n,t)-t^{-\theta}E^{1-\theta}_{1,1-\theta}(-\eta t)p(n,0)+p(n,x)l_{\eta,\theta}(x,t)\big|_{x=0}  \\
		&\ \ \ +\int_{0}^{\infty}\big(\Lambda(p(n-1,x)-p(n,x))-\bar{\Lambda}(p(n,x)-p(n+1,x))\big)l_{\eta,\theta}(x,t)\,\mathrm{d}x,
	\end{align*}
	where in the last step we have used \eqref{digsp} and in the penultimate step we have used $\lim_{x\to \infty}l_{\eta,\theta}(x,t)=0$ (see Alrawashdeh {\it et al.} (2017), Lemma 4.6). 
	
	We get our desired result on using (\ref{gfdt}).
\end{proof}
Next, we obtain another differential equation for the pmf $\bar{p}_{f_{2}}(n,t)$ for a special case of the stability index in inverse TSS.  Note that using forward and backward shift operator, \eqref{digsp} can be rewritten as
\begin{equation}\label{t=0}
	\frac{\mathrm{d}}{\mathrm{d}t}p(n,t)=\big(\Lambda(B-1)+\bar{\Lambda}(F-1)\big)p(n,t),
\end{equation}
where $B$ is the backward shift operator, that is, $Bp(n,t)=p(n-1,t)$ and $F$ is the forward shift operator, that is, $Fp(n,t)=p(n+1,t)$. Thus,
\begin{equation}\label{fb}
	\frac{\mathrm{d}^j}{\mathrm{d}t^j}p(n,t)=\big(\Lambda(B-1)+\bar{\Lambda}(F-1)\big)^jp(n,t), \ \ \forall j\ge 1.
\end{equation}
\begin{proposition}
For the integer valued stability index $m = 1/\theta \geq 2$ of inverse TSS, the pmf $\bar{p}_{f_{2}}(n,t)$, $n \in \mathbb{Z}$ solves 
	\begin{align*}
		\frac{\mathrm{d}}{\mathrm{d}t}\bar{p}_{f_2}(n,t) &=\sum_{j=1}^{m}\sum_{k=1}^{j} \binom{m}{j} \eta^{(1-j/m)} (-1)^{j+k} \frac{\mathrm{d}^{k-1}}{\mathrm{d}x^{k-1}} p(n,x) 
		\frac{\partial^{j-k}}{\partial x^{j-k}} l_{\eta,1/m}(x,t) \Big|_{x=0}
		\\& \ \ + \sum_{j=1}^{m}\binom{m}{j} \eta^{(1-j/m)}(-1)^j \big(\Lambda(B-1)+\bar{\Lambda}(F-1)\big)^j\bar{p}_{f_2}(n,t)- \delta_0(t)\bar{p}_{f_2}(n,0),
	\end{align*}
	where $l_{\eta,{1}/{m}}(x,t)$ is the density of the inverse TSS, $B$ is the backward shift operator and $F$ is the forward shift operator.
\end{proposition}
\begin{proof}
Note that, when $m = 1/\theta \geq 2$ is an integer, then the density $l_{\eta,{1}/{m}}(x,t)$ of the inverse TSS satisfies the following partial differential equation (see Kumar {\it et al.} (2011), Eq. (4.7)):
\begin{equation*}
	\frac{\partial}{\partial t}  l_{\eta,{1}/{m}}(x,t) + \delta_0(t) \delta_0(x) = \sum_{j=1}^{m}(-1)^j \binom{m}{j}\eta^{(1-{j}/{m})} \frac{\partial^{j}}{\partial x^{j}}  l_{\eta,{1}/{m}}(x,t), 
\end{equation*}
where $\delta_0(x) =  l_{\eta,{1}/{m}}(x,0)$. 	
	Here, $ l_{\eta,{1}/{m}}(x,t)$ is infinitely differentiable which follows from Theorem 3.1 of  Meerschaert and Scheffler (2008) and the fact that the density of TSS is infinitely differentiable. Also, we have $\lim_{x \to \infty}  l_{\eta,{1}/{m}}(x,t) = 0$ (see Alrawashdeh {\it et al.} (2017), Lemma 4.6).
	
	On using the above result in \eqref{gfdt} for $m = 1/\theta \geq 2$, we get
	\begin{align*}
		\frac{\mathrm{d}}{\mathrm{d}t} \bar{p}_{f_2}(n,t) &= \int_{0}^{\infty} p(n,x)\bigg( 
		\sum_{j=1}^{m} \binom{m}{j} \eta^{(1-j/m)}(-1)^j \frac{\partial^j}{\partial x^j} 
		l_{\eta,1/m}(x,t) - \delta_0(t)\delta_0(x) \bigg)\mathrm{d}x\\
		&= \sum_{j=1}^{m} \binom{m}{j} \eta^{(1-j/m)} (-1)^j 
		\int_{0}^{\infty} p(n,x) \frac{\partial^j}{\partial x^j} l_{\eta,1/m}(x,t) \mathrm{d}x 
		- \delta_0(t) p(n,0)\\
		&= \sum_{j=1}^{m} \binom{m}{j} \eta^{(1-j/m)} (-1)^j \times \Big( 
		\sum_{k=1}^{j} (-1)^k \frac{\mathrm{d}^{k-1}}{\mathrm{d}x^{k-1}} p(n,x) 
		\frac{\partial^{j-k}}{\partial x^{j-k}} l_{\eta,1/m}(x,t) \Big|_{x=0}
		\\& \ \ + (-1)^j \int_{0}^{\infty} l_{\eta,1/m}(x,t) \frac{\mathrm{d}^j}{\mathrm{d} x^j} p(n,x)\mathrm{d}x \Big)- \delta_0(t){p}(n,0)\\
		&=\sum_{j=1}^{m} \binom{m}{j} \eta^{(1-j/m)} (-1)^j \times \Big( 
		\sum_{k=1}^{j} (-1)^k \frac{\mathrm{d}^{k-1}}{\mathrm{d}x^{k-1}} p(n,x) 
		\frac{\partial^{j-k}}{\partial x^{j-k}} l_{\eta,1/m}(x,t) \Big|_{x=0}
		\\& \ \ + (-1)^j \big(\Lambda(B-1)+\bar{\Lambda}(F-1)\big)^j\int_{0}^{\infty} l_{\eta,1/m}(x,t) p(n,x)\mathrm{d}x \Big)- \delta_0(t){p}(n,0),
	\end{align*}
	where in the last step we have used \eqref{fb}. Thus, the proof is completed on using $\bar{p}_{f_2}(n,0)=p(n,0)$ and \eqref{gfdt}.

\end{proof}

\subsection{GSP time-changed by the first passage time of IGS}
Consider an IGS  $\{Y(t)\}_{t\geq0}$ with parameters $\delta > 0$, $\gamma> 0$ whose associated Bern\v stein function $f_{3}$ is given in (\ref{ploiuy67}). The first passage time $\{H(t)\}_{t\ge0}$ of IGS is defined as
\begin{equation*}
	H(t)\coloneqq \inf \{ r\ge 0 : Y(r)>t\},\ \ t\geq0.
\end{equation*}
Let us consider the following time-changed process:
\begin{equation}\label{inv}
	\bar{\mathcal{Z}}_{f_{3}}(t)\coloneqq \mathcal{S}(H (t)),\ \ t\ge0,
\end{equation}
where the GSP is independent of  $\{H(t)\}_{t\ge0}.$ 
\begin{proposition}
For $n\in \mathbb{Z}$, the pmf $\bar{p}_{f_3}(n,t)=\mathrm{Pr}\{\bar{\mathcal{Z}}_{f_{3}}(t)=n\}$   solves the following differential equation:
{\small	\begin{equation*}
		\delta\Big(\gamma^{2}+2\frac{\mathrm{d}}{\mathrm{d} t}\Big)^{1/2}\bar{p}_{f_{3}}(n,t)=\big(\delta\gamma-\Lambda-\bar{\Lambda}\big)\bar{p}_{f_{3}}(n,t)+\Lambda\bar{p}_{f_{3}}(n-1,t)+\bar{\Lambda}\bar{p}_{f_{3}}(n+1,t)-\delta\gamma \mathrm{Erf}\big(\gamma\sqrt{t/2}\big)\bar{p}_{f_3}(n,0),
	\end{equation*}}
where $\mathrm{Erf}(\cdot)$ is the error function.
	
	\begin{proof}
		From (\ref{inv}), we have
		\begin{equation}\label{gfdt1}
			\bar{p}_{f_{3}}(n,t)=\int_{0}^{\infty}p(n,x)h(x,t)\,\mathrm{d}x.
		\end{equation}
The following result will be used (see Wyloma\'nska {\it et al.} (2016), Eq. (2.22)):
\begin{equation*}
	\frac{\partial}{\partial x}h(x,t)=-\delta\Big(\gamma^{2}+2\frac{\partial}{\partial t}\Big)^{1/2}h(x,t)+\delta\gamma h(x,t)-\delta\sqrt{2/\pi t}e^{-\gamma^{2}t/2}\delta_0(x),
\end{equation*}
with the initial condition $h(x,0)=\delta_0(x)$. 
Using the above result in \eqref{gfdt1}, we get
{\small		\begin{align}\label{22}
			\delta\Big(\gamma^{2}+2\frac{\mathrm{d}}{\mathrm{d} t}\Big)^{1/2}\bar{p}_{f_{3}}(n,t)&=\int_0^{\infty} p(n,x)\Big(\delta\gamma h(x,t)-\delta\sqrt{2/\pi t}e^{-\gamma^{2}t/2}\delta_0(x)-\frac{\partial}{\partial x}h(x,t)\Big)\mathrm{d}x \nonumber\\
			&=\delta\gamma\bar{p}_{f_{3}}(n,t)-\delta\sqrt{2/\pi t}e^{-\gamma^{2}t/2}p(n,0)-\int_0^{\infty} p(n,x)\frac{\partial}{\partial x}h(x,t)\mathrm{d}x.
		\end{align}	}
Also, we use the following limiting results  (see Vellaisamy and Kumar (2018)):
\begin{equation}\label{55}
	\lim_{x \to \infty} \frac{\partial}{\partial x} h(x, t) = \lim_{x \to \infty} h(x, t) = 0,
\end{equation}
to obtain		
	\begin{align}\label{inthxt}
		\int_0^{\infty} &p(n,x)\frac{\partial}{\partial x}h(x,t)\mathrm{d}x\nonumber\\
		&=\big[p(n, x) h(x, t)\big]_0^\infty-\int_0^{\infty}h(x,t)\frac{\mathrm{d}}{\mathrm{d}x}p(n,x)\mathrm{d}x\nonumber\\
		&=-p(n,0)h(0,t)-\int_0^{\infty}\Big(\Lambda\big(p(n-1,x)-p(n,x)\big)-\bar{\Lambda}\big(p(n,x)-p(n+1,x)\big)\Big)h(x,t)\mathrm{d}x\nonumber\\
		&=-\bar{p}_{f_3}(n,0)h(0,t) - \Lambda\bar{p}_{f_3}(n-1,t)+\big(\Lambda+\bar{\Lambda}\big)\bar{p}_{f_3}(n,t)-\bar{\Lambda}\bar{p}_{f_3}(n+1,t),
	\end{align}
where in the last step we have used $\bar{p}_{f_3}(n,0)=p(n,0)$ and (\ref{gfdt1}), and in the penultimate step we have used \eqref{digsp}.
		
		 Finally, the proof  follows on substituting \eqref{inthxt} in \eqref{22} along with the following result (see Vellaisamy and Kumar (2018), Proposition 2.2):
		\begin{equation*}
			\lim\limits_{x\to0}h(x,t)=h(0,t)=\delta e^{-\gamma^{2}t/2}\Big(\sqrt{2/\pi t}-\gamma  e^{\gamma^{2}t/2} \mathrm{Erf}\big(\gamma\sqrt{t/2}\big)\Big).
		\end{equation*}  
	\end{proof}
\end{proposition}
Next, we obtain another governing differential equation for the pmf $\bar{p}_{f_3}(n,t)$ where the following results will be used (see Vellaisamy and Kumar (2018)):
\begin{align}
	\frac{\partial^2}{\partial x^2} h(x, t) 
	- 2\delta \gamma  \frac{\partial}{\partial x} h(x, t) 
	&= 2\delta^2 \frac{\partial}{\partial t} h(x, t) 
	+ 2\delta^2 h(x, 0) \delta_0(t),\label{pp}\\
	\lim_{x \to 0} \frac{\partial}{\partial x} h(x, t) &= 2\delta \gamma  h(0, t)\label{56}.
\end{align}   
\begin{proposition}
	The pmf $\bar{p}_{f_3}(n,t)$ satisfies the following:
	\begin{align*}
		\frac{\mathrm{d}}{\mathrm{d}t} \bar{p}_{f_3}(n, t) &=-\delta_o(t)\big)\bar{p}_{f_3}(n,0)+\frac{1}{2 \delta ^2}\big(\Lambda(B-1)+\bar{\Lambda}(F-1)\big)\bar{p}_{f_3}(n,0)h(0,t)\\
		&\ \ +\frac{1}{2 \delta ^2}\Big(\big(\Lambda(B-1)+\bar{\Lambda}(F-1)\big)\big(\big(\Lambda(B-1)+\bar{\Lambda}(F-1)\big)+2\delta\gamma\big)\Big)\bar{p}_{f_3}(n,t), \ \ n\in \mathbb{Z}. 
	\end{align*}
	\begin{proof} 
		From \eqref{gfdt1} and \eqref{pp}, we have
		\begin{align}\label{sc1}
			\frac{\mathrm{d}}{\mathrm{d}t} \bar{p}_{f_3}(n, t) &=\frac{1}{2\delta^2} \int_0^\infty p(n, x) \Big(\frac{\partial^2}{\partial x^2} h(x, t)
			-2\delta \gamma \frac{\partial}{\partial x} h(x, t)
			- 2\delta^2 h(x, 0) \delta_0(t)\Big)\mathrm{d}x\nonumber\\
			&= \frac{1}{2\delta^2} \int_0^\infty  p(n, x)\Big(\frac{\partial^2}{\partial x^2} h(x, t) 
			- 2\delta \gamma \frac{\partial}{\partial x} h(x, t) \Big) \mathrm{d}x - \delta_0(t) \bar{p}_{f_3}(n, 0). 
		\end{align}
From \eqref{fb}, we get		
			\begin{equation}\label{2d}
			\frac{\mathrm{d}^2}{\mathrm{d}x^2}p(n,x)=\big(\Lambda(B-1)+\bar{\Lambda}(F-1)\big)^2p(n,x).
		\end{equation}
		Also, for $t=0$, the equation \eqref{t=0} reduces to
		 \begin{equation}\label{t=00}
		 	\frac{\mathrm{d}}{\mathrm{d}t}p(n,t)\Big|_{t=0}=\big(\Lambda(B-1)+\bar{\Lambda}(F-1)\big)p(n,0).
		 \end{equation}
Observe that
		
		\begin{align}\label{sc2}
			\int_0^\infty &p(n, x) \frac{\partial^2}{\partial x^2} h(x, t) \,\mathrm{d}x\nonumber\\ 
			&=p(n, x) \frac{\partial}{\partial x} h(x, t) \Big|_0^\infty 
			- \int_0^\infty \frac{\mathrm{d}}{\mathrm{d}x} p(n, x) \frac{\partial}{\partial x} h(x, t) \mathrm{d}x \nonumber\\
			&= p(n, x) \frac{\partial}{\partial x} h(x, t) \Big|_0^\infty 
			- h(x, t) \frac{\mathrm{d}}{\mathrm{d} x} p(n, x) \Big|_0^\infty 
			+ \int_0^\infty h(x, t) \frac{\mathrm{d}^2}{\mathrm{d} x^2} p(n, x) \mathrm{d}x\nonumber\\
			&=-2\delta \gamma p(n,0)h(0,t)+\big(\Lambda(B-1)+\bar{\Lambda}(F-1)\big)p(n,0)h(0,t)\nonumber\\
			& \ \ + \big(\Lambda(B-1)+\bar{\Lambda}(F-1)\big)^2\int_{0}^\infty p(n,x)h(x,t)\mathrm{d}x, \ (\text{using}\ \eqref{55}, \eqref{56}, \eqref{2d}, \eqref{t=00} )\nonumber\\
			&=-2\delta \gamma h(0,t) \bar{p}_{f_3}(n,0)+\big(\Lambda(B-1)+\bar{\Lambda}(F-1)\big)\bar{p}_{f_3}(n,0)h(0,t)\nonumber\\
			&\ \ +\big(\Lambda(B-1)+\bar{\Lambda}(F-1)\big)^2\bar{p}_{f_3}(n,t),
		\end{align}
		where in the last step we have used $\bar{p}_{f_3}(n,0)=p(n,0)$. Moreover, from \eqref{inthxt}, we have
		\begin{equation}\label{sc3}
		\int_0^\infty p(n, x) \frac{\partial}{\partial x} h(x, t) \mathrm{d}x= -\bar{p}_{f_3}(n,0)h(0,t)-\big(\Lambda(B-1)+\bar{\Lambda}(F-1)\big)\bar{p}_{f_3}(n,t).
		\end{equation}
Finally, on substituting \eqref{sc2} and \eqref{sc3} in \eqref{sc1}, we get our desired result.
	\end{proof}
\end{proposition}

\end{document}